\documentclass{amsart}

\usepackage{amsmath, amssymb, amsthm, amsfonts, amsxtra, amscd}

\input xy
\xyoption{all}

\usepackage{hyperref}

\numberwithin{equation}{section}

\theoremstyle{plain}
\newtheorem{theorem}[subsection]{Theorem}
\newtheorem{lemma}[subsection]{Lemma}
\newtheorem{prop}[subsection]{Proposition}
\newtheorem{cor}[subsection]{Corollary}

\theoremstyle{definition}
\newtheorem{defn}[subsection]{Definition}

\newtheorem{remark}[subsection]{Remark}
\newtheorem{exam}[subsection]{Example}

\def\AA{\mathbb{A}}
\def\BB{\mathbb{B}}
\def\CC{\mathbb{C}}
\def\DD{\mathbb{D}}

\def\FF{\mathbb{F}}
\def\GG{\mathbb{G}}

\def\PP{\mathbb{P}}
\def\QQ{\mathbb{Q}}

\def\SS{\mathbb{S}}
\def\TT{\mathbb{T}}

\def\WW{\mathbb{W}}

\def\ZZ{\mathbb{Z}}

\def\calA{\mathcal{A}}
\def\calB{\mathcal{B}}

\def\calD{\mathcal{D}}
\def\calE{\mathcal{E}}

\def\calI{\mathcal{I}}

\def\calL{\mathcal{L}}

\def\calN{\mathcal{N}}
\def\calO{\mathcal{O}}
\def\calP{\mathcal{P}}
\def\calQ{\mathcal{Q}}

\def\bH{\mathbf{H}}
\def\bI{\mathbf{I}}

\def\bK{\mathbf{K}}

\def\bP{{\mathbf{P}}}
\def\bQ{\mathbf{Q}}
\def\bR{\mathbf{R}}

\newcommand{\tilw}{\widetilde{w}}
\newcommand{\tilW}{\widetilde{W}}
\newcommand{\tilC}{\widetilde{C}}

\newcommand{\tilf}{\widetilde{f}}
\newcommand{\tL}{\widetilde{L}}
\newcommand{\tX}{\widetilde{X}}

\newcommand{\fra}{\mathfrak{a}}

\newcommand{\frG}{\mathfrak{G}}
\newcommand{\frA}{\mathfrak{A}}

\newcommand{\dualG}{\widehat{G}}
\newcommand{\dualT}{\widehat{T}}

\newcommand{\dualH}{\widehat{H}}
\newcommand\dualg{{\widehat{\mathfrak{g}}}}

\newcommand\aff{\textup{aff}}
\newcommand{\AS}{\textup{AS}}

\newcommand{\Bl}{\textup{Bl}}
\newcommand{\Bun}{\textup{Bun}}

\newcommand{\Dyn}{\textup{Dyn}}
\newcommand\ev{\textup{ev}}

\newcommand{\Fl}{\textup{Fl}}
\newcommand\Four{\textup{Four}}

\newcommand\Gal{\textup{Gal}}

\newcommand{\Gr}{\textup{Gr}}

\newcommand{\Hk}{\textup{Hk}}
\newcommand{\IC}{\textup{IC}}
\newcommand\id{\textup{id}}

\newcommand{\Kl}{\textup{Kl}}
\newcommand\Lie{\textup{Lie}}

\newcommand\opp{\textup{opp}}
\newcommand\Out{\textup{Out}}
\newcommand\Perv{\textup{Perv}}

\newcommand{\pr}{\textup{pr}}

\newcommand\Rep{\textup{Rep}}
\newcommand{\Res}{\textup{Res}}

\newcommand\rs{\textup{rs}}

\newcommand\Spec{\textup{Spec}}
\newcommand\St{\textup{St}}

\newcommand\st{\textup{st}}

\newcommand{\Swan}{\textup{Swan}}
\newcommand\Sym{\textup{Sym}}

\newcommand{\Vect}{\textup{Vect}}

\newcommand\Aut{\textup{Aut}}
\newcommand\Hom{\textup{Hom}}
\newcommand\End{\textup{End}}

\newcommand{\unc}{\underline{c}}
\newcommand{\unu}{\underline{u}}
\newcommand{\unw}{\underline{w}}

\newcommand\GL{\textup{GL}}

\newcommand\Ug{\textup{U}}

\newcommand\SO{\textup{SO}}
\newcommand\Og{\textup{O}}

\newcommand\Sp{\textup{Sp}}

\newcommand{\Gm}{\GG_m}

\newcommand{\ad}{\textup{ad}}
\newcommand{\Ad}{\textup{Ad}}
\def\sc{\textup{sc}}
\newcommand{\der}{\textup{der}}

\newcommand\xch{\mathbb{X}^*}
\newcommand\xcoch{\mathbb{X}_*}

\newcommand{\isom}{\stackrel{\sim}{\to}}
\newcommand{\incl}{\hookrightarrow}
\newcommand{\surj}{\twoheadrightarrow}

\newcommand{\ep}{\epsilon}
\renewcommand{\l}{\lambda}
\renewcommand{\L}{\Lambda}
\newcommand{\chk}{\textup{char}(k)}
\newcommand{\oll}{\overleftarrow}
\newcommand{\orr}{\overrightarrow}

\newcommand{\leftexp}[2]{{\vphantom{#2}}^{#1}{#2}}
\newcommand{\pH}{\leftexp{p}{\textup{H}}}

\newcommand{\Ql}{\overline{\QQ}_\ell}
\newcommand{\const}[1]{\overline{\QQ}_{\ell,#1}}

\newcommand{\cohog}[2]{\textup{H}^{#1}({#2})}     
\newcommand{\cohoc}[2]{\textup{H}_{c}^{#1}({#2})} 
\newcommand{\upH}{\textup{H}}
\newcommand{\conv}[1]{\stackrel{#1}{*}}
\newcommand{\jiao}[1]{\langle{#1}\rangle}
\newcommand{\wt}[1]{\widetilde{#1}}
\newcommand\mat[4]{\left(\begin{array}{cc} #1 & #2 \\ #3 & #4 \end{array}\right)}  
\newcommand{\quash}[1]{}

\newcommand{\tKl}{\wt{\Kl}}
\newcommand{\pline}{X^{\circ}}
\newcommand{\tpline}{\tX^{\circ}}

\newcommand{\Lab}{L^{\textup{ab}}}
\newcommand{\tLab}{\widetilde{L}^{\textup{ab}}}
\newcommand{\Pss}{\bP^{\textup{ss}}_0}
\newcommand{\Gss}{\Gamma^{\textup{ss}}_0}
\newcommand{\rot}{\textup{rot}}
\newcommand{\grot}{\Gm^{\rot}}
\newcommand{\bcP}{\overline{\calP}}
\newcommand{\dG}{\leftexp{L}{G}}

\newcommand{\bark}{\bar{k}}
\newcommand{\barZ}{\bar{Z}}

\newcommand{\Wa}{W_{\textup{aff}}}

\title{Epipelagic representations and rigid local systems}
\author{Zhiwei Yun}
\email{zyun@stanford.edu}
\address{Department of Mathematics, Stanford University, 450 Serra Mall, Stanford, CA 94305}
\thanks{Supported by the NSF grant DMS-1302071 and Packard Fellowship.}
\date{}

\begin{document}

\begin{abstract}
We construct automorphic representations for quasi-split groups $G$ over the function field $F=k(t)$ one of whose local components is an epipelagic representation in the sense of Reeder and Yu. We also construct the attached Galois representations under the Langlands correspondence. These Galois representations give new classes of conjecturally rigid, wildly ramified $\dG$-local systems over $\PP^{1}-\{0,\infty\}$ that  generalize the Kloosterman sheaves constructed earlier  by Heinloth, Ng\^o and the author. We study the monodromy of these local systems and compute all examples when $G$ is a classical group.  
\end{abstract}

\maketitle

\tableofcontents

\section{Introduction}

\subsection{The goal} Let $G$ be a reductive quasi-split group over a local field $K$ as in \S\ref{sss:qs}. Recently, Reeder and Yu \cite{RY} constructed a family of supercuspidal representations of $p$-adic groups called {\em epipelagic representations}, generalizing the {\em simple supercuspidals} constructed earlier by Gross and Reeder \cite{GR}. These are supercuspidal representations constructed by compactly inducing certain characters of the pro-$p$ part of a parahoric subgroup $\bP$ of $G$. The construction of epipelagic representations uses the $\theta$-groups $(L_{\bP}, V_{\bP})$ studied by Vinberg {\em et al}. 


The main results of this paper include
\begin{itemize}
\item Realization of epipelagic representations (for the local function field) as a local component of an automorphic representation $\pi$ of $G(\AA_{F})$, where $F$ is the function field $k(t)$. This is done in Proposition \ref{p:uniquefunction}. Here $\pi=\pi(\chi,\phi)$ depends on two parameters $\chi$ and $\phi$ (a multiplicative on $\tL_{\bP}(k)$ and a {\em stable} linear functional $\phi: V_{\bP}\to k$). 

\item Construction of the Galois representation $\rho_{\pi}:\Gal(F^{s}/F)\to\dG(\Ql)$ attached to $\pi=\pi(\chi,\phi)$ under the Langlands correspondence.  This is done in Theorem \ref{th:main} and Corollary \ref{c:Klphi}. The Galois representation $\rho_{\pi}$ can be equivalently thought of as an $\ell$-adic $\dG$-local systems $\Kl_{\dG, \bP}(\chi, \phi)$ over $\PP^{1}-\{0,\infty\}$, where $\dG$ is the Langlands dual group of $G$. We also offer a way to calculate these $\dG$-local systems in terms of the Fourier transform (see Proposition \ref{p:Four}).

\item Description of the local monodromy the local systems $\Kl_{\dG, \bP}(\chi, \phi)$. The main result is Theorem \ref{th:unip} that describes the monodromy of $\Kl_{\dG, \bP}(\chi, \phi)$ at $0$ when $G$ is split and $\chi=1$. We also conditionally deduce the cohomological rigidity of $\Kl_{\dG, \bP}(\chi,\phi)$ (see \S\ref{s:rigidity}), and make predictions on the monodromy at $0$ in general (see \S\ref{ss:localmono}).

\item Computation of the local systems $\Kl_{\dG, \bP}(\chi, \phi)$ when $G$ is a classical group (\S\ref{s:u}, \S\ref{s:sp} and \S\ref{s:o}).  In each of these cases, we express the local system $\Kl_{\dG, \bP}(\chi, \phi)$ as the Fourier transform of the direct image complex of an explicit morphism to $V_{\bP}$.
As a result, we obtain new families of exponential sums (indexed by $G$ and $\bP$) generalizing Kloosterman sums, see Corollaries \ref{c:uKl}, \ref{c:spKl} and \ref{c:oKl}.
\end{itemize}

\subsection{Comparison with \cite{HNY}}
In \cite{HNY}, Heinloth, Ng\^o and the author considered the case of ``simple supercuspidals'' of Gross and Reeder. These correspond to the special case $\bP=\bI$ is an Iwahori subgroup. In \cite{HNY} we constructed the Kloosterman sheaves $\Kl_{\dualG, \bI}(\chi,\phi)$, and proved properties of their local and global monodromy expected by Frenkel and Gross in \cite{FG} for split $G$.

The general case to be considered in this article exhibits certain interesting phenomena which are not seen in the special case treated in \cite{HNY}. 

First, the $\dualG$-local systems $\Kl_{\dualG, \bP}(\chi,\phi)$ on $\PP^{1}-\{0,\infty\}$ form an algebraic family as the additive character $\phi$ defining $\pi_{\infty}$ varies. In other words, these local systems are obtained from a single ``master'' $\dualG$-local system $\Kl_{\dualG,\bP}(\chi)$ on the larger base $V^{*,\st}_{\bP}$ by restriction to various $\Gm$-orbits.  Here the base space $V^{*,\st}_{\bP}$ is the stable locus of the dual vector space of $V_{\bP}$ which is part of Vinberg's $\theta$-groups.

Second, when $\chi=1$, the local system $\Kl_{\dualG, \bP}(1,\phi)$ has unipotent tame monodromy at $0$ given by a unipotent class $\unu$ in $\dualG^{\sigma,\circ}$ ($\sigma$ is the pinned automorphism defining $G$), which only depends on the type of $\bP$. On the other hand, the types of $\bP$ are in bijection with regular elliptic $\WW$-conjugacy classes $\unw$ in $\WW\sigma$ (at least when $\chk$ is large, see \S\ref{ss:bij}). Therefore our construction gives a map
\begin{equation}\label{wu}
\{\mbox{regular elliptic $\WW$-conjugacy classes in $\WW\sigma$}\}\to\{\mbox{unipotent classes in $\dualG^{\sigma,\circ}$}\}
\end{equation}
sending $\unw$ to $\unu$ via the intermediate step of an admissible parahoric subgroup $\bP$.  When $G$ is split, Theorem \ref{th:unip} gives an alternative description of this map using Lusztig's theory of cells in affine Weyl groups, and using this we are able to compute the map \eqref{wu} for all types of $G$ in \S\ref{ss:tables} (verified for split $G$ and conjectural in general). In \cite{L}, Lusztig defined a map from all conjugacy classes in $\WW$ to unipotent conjugacy classes in $\dualG$. One can check case-by-case that this map coincides with the restriction of Lusztig's map to regular elliptic conjugacy classes.

\subsection{Open questions} As in the work of Frenkel and Gross \cite{FG}, we expect that there should be a parallel story when $\ell$-adic local systems are replaced with connections on algebraic vector bundles (on varieties over $\CC$). In particular, the ``master'' $\dualG$-local system $\Kl_{\dualG,\bP}(\chi)$ should correspond to a $\dualG$-connection over $V^{*,\st}_{\bP,\CC}$. When $G$ is a classical group, the formulae in \S\ref{s:u}-\S\ref{s:o} give descriptions of these connections as Fourier transform of Gauss-Manin connections. Are there simple formulae for these connection in general?

Another problem is to calculate the Euler characteristics (equivalently Swan conductor at $\infty$) of the local systems $\Kl^{V}_{\dualG,\bP}(1,\phi)$ for representations $V$ of $\dualG$. Such calculations would give evidence to (and sometimes proofs of) the prediction made by Reeder and Yu about the Langlands parameters of epipelagic representations (see \S\ref{ss:wildL}). We do one such calculation in the case $G$ is a unitary group (see Proposition \ref{p:uchi}), but in general the complexity of the calculation seems to be formidable (see also the proof of \cite[Theorem 4]{HNY} in which we treated the case $\bP$ is Iwahori, which already took many pages). 

While we have a more or less complete picture for the local monodromy of the local systems $\Kl_{\dualG,\bP}(1,\phi)$ (partly conjectural), we do not discuss their global monodromy here, i.e., the Zariski closure of the image of the geometric $\pi_{1}(\PP^{1}-\{0,\infty\})$ in $\dualG$. It can be as small as a finite group as we see in Proposition \ref{p:um2}(2) when $G$ is an odd unitary group and $\bP$ is a special parahoric subgroup. Does the global monodromy group of $\Kl_{\dualG,\bP}(1,\phi)$ depend only on $\bP$, and if so, how do we read it off from $\bP$?

\subsection{Notation and convention}

\subsubsection{The function field} Let $k$ be a finite field and $F=k(t)$ be the rational function field over $k$. Places of $F$ are in natural bijection with closed points of $X:=\PP^{1}_{k}$, the set of which is denoted by $|X|$. In particular we have two places $0$ and $\infty$ of $F$. For a place $x\in|X|$, let $F_{x}$ (resp. $\calO_{x}$) be the completed local field (resp. completed local ring) of $X$ at $x$, and let $k_{x}$ be the residue field at $x$. Let $\AA_{F}=\prod'_{x\in |X|}F_{x}$ be the ring of ad\`eles of $F$.

\subsubsection{Sheaves} Let $\ell$ be a prime number different from $\chk$. We shall consider constructible $\Ql$-complexes over various algebraic stacks over $k$ or $\bark$. All sheaf-theoretic operations are understood as {\em derived functors}.
 
\subsubsection{The absolute group data} Let $\GG$ be a split reductive group over $k$ whose derived group is absolutely almost simple. Fix a pinning $\dagger=(\BB, \TT,\cdots)$ of $\GG$, where $\BB$ is a Borel subgroup of $\GG$ and $\TT\subset\BB$ a split maximal torus. Let $\WW=N_{\GG}(\TT)/\TT$ be the Weyl group of $\GG$. Let $\Phi\subset\xch(\TT)$ be the set of roots. Fix a cyclic subgroup $\ZZ/e\ZZ\incl\Aut^{\dagger}(\GG)$ of the pinned automorphism group of $\GG$, and denote the image of $1$ by $\sigma\in\Aut^{\dagger}(\GG)$. 

We assume $\chk$ is prime to $e$ and that $k^{\times}$ contains $e$th roots of unity. 

Let $Z\GG$ be the center of $\GG$. We also assume that $\xch(Z\GG)^{\sigma}=0$.

\subsubsection{The quasi-split group}\label{sss:qs} Let $\tX\to X$ be the $\mu_{e}$-cover totally ramified over $0$ and $\infty$. Then $\tX$ is also isomorphic to $\PP^{1}_{k}$, and we denote its affine coordinate by $t^{1/e}$. We denote
\begin{equation*}
\pline:=X-\{0,\infty\} ;\hspace{1cm} \tpline:=\tX-\{0,\infty\}.
\end{equation*}

The data $\GG$ and $\sigma$ define a quasi-split group scheme $G$ over $\pline$ which splits over the $\mu_{e}$-cover $\tpline\to\pline$. More precisely, for any $k[t,t^{-1}]$-algebra $R$, we have $G(R)=\GG(R\otimes_{k[t,t^{-1}]}k[t^{1/e}, t^{-1/e}])^{\mu_{e}}$ where $\mu_{e}$ acts on $t^{1/e}$ by multiplication and on $\GG$ via the fixed map $\mu_{e}\incl\Aut^{\dagger}(\GG)$. Since $\xch(Z\GG)^{\sigma}=0$, the center of $G$ does not contain a split torus. 

Let $\SS$ be the neutral component of $\TT^{\sigma}$. Then $S=\SS\otimes_{k}F$ is a maximal split torus of $G$ over $F$.

\subsubsection{Langlands dual group} Let $\dualG$ be the reductive group over $\Ql$ whose root system is dual to that of $\GG$. We also fix a pinned $\dagger$ of $\dualG$, through which we identify $\Aut^{\dagger}(\dualG)$ with $\Aut^{\dagger}(\GG)$. We define the Langlands dual group $\dG$ of $G$ to be $\dG=\dualG\rtimes\mu_{e}$ where $\mu_{e}$ acts through $\mu_{e}\incl\Aut^{\dagger}(\GG)\cong\Aut^{\dagger}(\dualG)$.

\subsubsection{Loop groups}\label{sss:loop}
We shall use the convention in \cite[\S2.2]{Ymotive} for loop groups. For example, when $x\in X$, we use $L_{x}G$ (resp. $L^{+}_{x}G$, if $x\in \pline$) to denote the loop group (resp.  positive loop group) of the group scheme $G$ at $x$. Then $L_{x}G$ (resp. $L_{x}^{+}G$) is a group ind-scheme (resp. pro-algebraic group) over $k(x)$, the residue field of $x$. A parahoric subgroup of $G(F_{x})$ is also viewed as a pro-algebraic subgroup of $L_{x}G$ over $k(x)$.  Since $G$ splits over $\tpline$, for $\wt{x}\in\tpline$, we use $L_{\wt{x}}\GG$ to denote the loop group of the constant group scheme $G\times_{\pline}\tpline\cong \GG\times \tpline$ at $\wt{x}$. 

\section{Epipelagic representations and automorphic representations}
In this section we let $K=F_{\infty}$ denote the local field of $F$ at $\infty$. We first recall the construction of epipelagic representations of $G(K)$ following Reeder and Yu \cite{RY}. We then realize these epipelagic representations as the local components  at $\infty$ of automorphic representations of $G(\AA_{F})$.

\subsection{Admissible parahorics and $\theta$-groups}
The Borel subgroup $\BB\subset\GG$ allows us to define a standard Iwahori subgroup $\bI\subset G(K)$ in the following way. Let $K_{e}$ be the totally ramified extension of $K$ of degree $e$, then $G(K)\subset \GG(K_{e})$ by construction. We define $\bI$ to be the preimage of $\BB$ under the homomorphism $G(K)\cap \GG(\calO_{K_{e}})\to\GG(\calO_{K_{e}})\surj\GG$, the last map being given by $\calO_{K_{e}}\surj k$. 

Let $\bP\subset G(K)$ be a standard parahoric subgroup (i.e., $\bI\subset\bP$). Let $\bP\supset\bP^{+}\supset\bP^{++}$ be the first three steps in the Moy-Prasad filtration of $\bP$. In particular, $\bP^{+}$ is the pro-unipotent radical of $\bP$ and $L_\bP:=\bP/\bP^{+}$ is the Levi factor of $\bP$. We may view $L_{\bP}$ as a subgroup of $\bP$ containing $\SS=\TT^{\sigma,\circ}$. Let $V_{\bP}=\bP^{+}/\bP^{++}$, a vector space over $k$ on which $L_{\bP}$ acts. The pair $(L_{\bP}, V_{\bP})$ is an example of a $\theta$-group in the terminology of the Vinberg school.

The parahoric $\bP$ is called {\em admissible} if there exists a closed orbit of $L_\bP$ on the dual space $V^{*}_\bP$ with finite stabilizers. Such orbits are called {\em stable}. Stable orbits form an open subset $V^{*,\st}_\bP\subset V^{*}_\bP$. Elements in $V^{*,\st}_{\bP}$ are called stable functionals on $V_{\bP}$. For a complete list of admissible parahoric subgroups for various types of $G$ when $\chk$ is large, see the tables in \cite[\S7.1-7.2]{GLRY}.

\subsection{Relation with regular elliptic numbers}\label{ss:bij}
Let $\WW'=\WW\rtimes\mu_{e}$ where $\mu_{e}$ acts on $\WW$ via its pinned action on $\GG$. Springer \cite{Spr} defined the notion of regular elements for $(\WW,\sigma)$. We shall use a slightly different notion of $\ZZ$-regularity defined by Gross, Levy, Reeder and Yu in \cite[Definition 1]{GLRY}. An element $w\in \WW'$ is {\em $\ZZ$-regular} if it permutes the roots $\Phi$ freely. See \cite[Proposition 1]{GLRY} for the relation between $\ZZ$-regularity and Springer's notion of regularity. Combining \cite[Proposition 1]{GLRY} and \cite[Proposition 6.4(iv)]{Spr} one easily sees that a $\ZZ$-regular element in $\WW\sigma\subset\WW'$ is determined, up to $\WW$-conjugacy, by its order.

On the other hand, an element $w\in \WW'$ is {\em elliptic} if $\xch(\TT^{\ad})^{w}=0$. The order of a $\ZZ$-regular elliptic element in $\WW\sigma\subset \WW'$ is called a {\em regular elliptic number} of the pair $(\WW,\sigma)$. The above discussion says that the assignment $w\mapsto \textup{ord}(w)$ is a bijection between $\ZZ$-regular elliptic elements  in $\WW\sigma$ up to $\WW$-conjugacy and regular elliptic number of the pair $(\WW,\sigma)$.

To each standard parahoric $\bP$, we can assign a natural number $m$ as follows. Consider the apartment $\frA$ in the building of $G(K)$ corresponding to the maximal split torus $S=\SS\otimes_{k}F$. Let $\xi$ be the barycenter of the facet corresponding to $\bP$. The number $m=m(\bP)$ is defined as the smallest positive integer such that $\alpha(\xi)\in\frac{1}{m}\ZZ$ for all affine roots $\alpha$ of $G(K)$ with respect to $S$ (see \cite[\S3.3]{RY}). 

When $\bP$ is admissible, $\chk$ is not torsion for $G$ and $\chk\nmid m(\bP)$, it was shown in \cite[Corollary 5.1]{RY} that $m(\bP)$ is a regular elliptic number of $(\WW,\sigma)$, and conversely every regular elliptic number $m$ determines an admissible parahoric $\bP$ up to conjugacy. Combing these facts, one sees that when $\chk$ is larger than the twisted Coxeter number $h_{\sigma}$ of $(\GG,\sigma)$, there are natural bijections between the following three sets
\begin{enumerate}
\item Admissible standard parahorics $\bP$ of $G(K)$;
\item $\ZZ$-regular elliptic elements  $w\in \WW\sigma$ up to $\WW$-conjugacy;
\item Regular elliptic numbers $m$ of the pair $(\WW, \sigma)$.
\end{enumerate}

For a list of regular elliptic numbers in all types of $G$, see \S\ref{ss:tables}.

\begin{exam}\label{ex:adm}
\begin{enumerate}
\item The Iwahori $\bI$ is always admissible. It corresponds to the twisted Coxeter number $h_{\sigma}$. In this case $L_{\bI}=\TT^{\sigma,\circ}$, and $V_{\bI}$ is the sum of affine simple root spaces of $G$.
\item Let $\sigma\in\Out(\GG)$ by the opposition of the Dynkin diagram. Then $-1\in\WW\sigma$ is a regular elliptic element, and $m=2$ is a regular elliptic number. The corresponding admissible $\bP_{2}$ is a maximal parahoric except in type $A_1$ and $C_n$. The Levi quotient $L_2\cong \GG^{\theta, \circ}$ where $\theta$ is a Chevalley involution on $\GG$, and $V_\bP$ can be identified with the $(-1)$-eigenspace of $\theta$ on $\Lie\GG$.
\item Besides the case $m=2$, there are a few more cases of admissible parahorics that are maximal:
\begin{equation*}
\leftexp{3}{D}_{4} \ (m=3), E_{6} \ (m=3), \leftexp{2}{E}_{6} \ (m=4), E_{8} \ (m=3,4,5) \textup{ and }F_{4}\  (m=3).
\end{equation*}

\item\label{subreg} There are a few cases of admissible parahorics that are minimal (but not Iwahori). They are
\begin{equation*}
\leftexp{2}{A}_{2},\leftexp{2}{A}_{5}, C_2,D_4,\leftexp{3}{D}_{4},E_6,\leftexp{2}{E}_6, E_7, E_8, F_4, G_2.
\end{equation*}
In these cases, the admissible minimal parahoric corresponds to the unique vertex in the affine Dynkin diagram $\wt\Dyn(G)$ with three or more edges coming towards it. 
\end{enumerate}
\end{exam}

From now on, till the end of \S\ref{s:rigidity}, we shall fix an admissible standard parahoric subgroup $\bP$ of $G(K)$. Let $m=m(\bP)$ as defined in \S\ref{ss:bij}. 

\subsection{Epipelagic supercuspidal representations}\label{ss:epi} Fix a nontrivial character $\psi:k=\FF_q\to\QQ_\ell(\mu_p)^\times$. For a stable functional $\phi\in V^{*,\st}_{\bP}(k)$, viewed as a linear map $\phi: V_{\bP}\to k$, Reeder-Yu \cite[Proposition 2.4]{RY} show that the compact induction
\begin{equation*}
\textup{c-Ind}^{G(K)}_{\bP^{+}}(\psi\phi)
\end{equation*}
is a finite direct sum of irreducible supercuspidal representations of $G(K)$ (recall the center of $G(K)$ is finite). Its simple summands are called {\em epipelagic representations} of $G(K)$ attached to the parahoric $\bP$ and the stable functional $\phi$. These irreducible representations of $G(K)$ are characterized by having an $\psi\phi$-eigenvector under $\bP^{+}$ and having depth $1/m$.

\subsection{Expected Langlands parameter}\label{ss:wildL}
Let $\calI_{K}\lhd W(K^{s}/K)$ be the inertia group and the Weil group of $K$. Let $\pi$ be an epipelagic representation of $G(K)$ attached to $\bP$ and a stable linear functional $\phi: V_{\bP}\to k$. According the the local Langlands conjecture, there should be a Galois representation $\rho_{\pi}:W(K^{s}/K)\to \dG(\Ql)$ attached to $\pi$ as the conjectural Langlands parameter. In particular, $\calI_{K}$ acts on the Lie algebra $\dualg$ of $\dualG$ by composing $\rho_{\pi}$ with the adjoint representation of $\dualG$. Guided by the conjecture in \cite{GR} relating the adjoint gamma factors and formal degrees, Reeder and Yu predicted in \cite[\S7.1]{RY} that\begin{enumerate}
\item $\dualg^{\calI_{K}}=0$;
\item $\Swan(\dualg)=\#\Phi/m$.
\end{enumerate}
If moreover $\chk$ does not divide the order of $\WW$, Reeder and Yu made even sharper predictions: the restriction of $\rho_{\pi}$ to $\calI_{K}$ should look like
\begin{equation*}
\xymatrix{1\ar[r] & \calI^+_{K} \ar[r]\ar[d] & \calI_{K}\ar[r]\ar[d]^{\rho_{\pi}} & \calI^t_{K}\ar[r]\ar[d] & 1\\
1\ar[r] & \dualT\ar[r] & N_{\dG}(\dualT)\ar[r] & \WW'\ar[r] & 1}
\end{equation*}
where such that a generator of the tame inertia $\calI^{t}_{K}$ maps to a regular elliptic element $w\in \WW\sigma$ of order $m$.

\subsection{Realization in automorphic representations}\label{ss:auto}
Recall $K=F_{\infty}$ and now we denote $\bP$ by $\bP_{\infty}$. Let $\bI_0\subset G(F_0)$ be the Iwahori subgroup corresponding to the opposite Borel $\BB^{\opp}$ of $\GG$, as $\bI=\bI_{\infty}$ was constructed from $\BB$. Let $\bP_{0}\subset G(F_{0})$ be the parahoric subgroup containing $\bI_0$ of the same type as $\bP_{\infty}$. Let $\wt\bP_{0}$ be the normalizer of $\bP_{0}$ in $ G(F_0)$ and $\Omega_{\bP}=\wt\bP_{0}/\bP_{0}$ \footnote{One can show that $\Omega_{\bP}$ is always equal to the full stabilizer $\Omega$ of the fundamental alcove under the extended affine Weyl group of $G(K)$. In other words, the sub-Dynkin diagram of the affine Dynkin diagram $\wt\Dyn(G)$ of $G$ representing an admissible standard parahoric $\bP$ is always stable under the automorphisms of $\wt{\Dyn}(G)$ fixing the extended node.}. The Levi quotient of $\bP_{0}$ is again identified with $L_{\bP}$. Let $\Lab_{\bP}=L_{\bP}/L_{\bP}^{\der}$ be the maximal quotient torus of $L_{\bP}$. Let $\Pss=\ker(\bP_{0}\to \Lab_{\bP})$ and let
\begin{equation*}
\tLab_{\bP}=\wt\bP_{0}/\Pss.
\end{equation*}
Note that $\tLab_{\bP}$ is an extension of the finite group $\Omega_{\bP}$ by the torus $\Lab_{\bP}$. 

Fix a character $\chi:\tLab_\bP(k)\to\Ql^\times$ and a {\em stable} linear functional $\phi:V_{\bP}\to k$.

Now we try to classify automorphic representations $\pi=\otimes'_{x\in|X|}\pi_{x}$ of $G(\AA_{F})$ satisfying
\begin{itemize}
\item $\pi_x$ is unramified for $x\neq 0,\infty$;
\item $\pi_0$ has an eigenvector under $\wt\bP_{0}$ on which it acts through $\chi$ via the quotient $\tLab_{\bP}$; 
\item $\pi_\infty$ has an eigenvector under $\bP^{+}_\infty$ on which it acts through $\psi\phi$ (i.e., $\pi_{\infty}$ is an epipelagic supercuspidal representation attached to $\bP^{+}_{\infty}$ and the stable functional $\phi$).
\end{itemize}

\begin{prop}\label{p:uniquefunction} There is a unique automorphic representation $\pi=\pi(\chi,\phi)$ of $G(\AA_{F})$ satisfying all the above conditions. Moreover, we have 
\begin{enumerate}
\item $\pi$ is cuspidal, and appears with multiplicity one in the automorphic spectrum of $G$.
\item Both eigenspaces $\pi_{0}^{(\wt\bP_{0},\chi)}$ and $\pi_{\infty}^{(\bP_{\infty}^{+},\psi\phi)}$ are one-dimensional.
\end{enumerate}
\end{prop}
\begin{proof} Consider the vector space of $\Ql$-valued functions
\begin{equation*}
\calA:=\textup{Fun}\left( G(F)\backslash G(\AA_F)/\prod_{x\neq0,\infty} G(\calO_x)\right)^{(\wt\bP_{0},\chi)\times(\bP^{+}_\infty,\psi\phi)}.
\end{equation*}
Here the superscript means taking the eigenspace under $\wt\bP_{0}\times\bP^{+}_{\infty}$ on which $\wt\bP_{0}$ acts through $\chi$ and $\bP^{+}_{\infty}$ acts through $\psi\phi$. We first show that $\dim\calA=1$.

Let $\wt\Gamma_{0}= G(k[t,t^{-1}])\cap\wt\bP_{0}$. Then $\wt{L}_{\bP}$ is also a subgroup of $\wt\Gamma_{0}$. Let  $W_{\bP}$ be the Weyl group of $L_{\bP}$, identified with a subgroup of the extended affine Weyl group $\tilW$ of $G(F_{\infty})$ (both are with respect to $\SS$).  Let $\tilW_{\bP}=N_{\tilW}(W_{\bP})$, which is an extension of $\Omega_{\bP}$ by $W_{\bP}$.

A parahoric variant of \cite[Proposition 1.1]{HNY} gives an equality between double cosets
\begin{equation*}
G(F)\backslash G(\AA_F)/(\wt\bP_{0}\times\prod_{x\neq0,\infty} G(\calO_x)\times\bP^{+}_\infty)=\wt\Gamma_{0}\backslash G(F_\infty)/\bP^{+}_{\infty}
\end{equation*}
and the Birkhoff decomposition
\begin{equation*}
G(F_\infty)=\bigsqcup_{\tilw\in \tilW_{\bP}\backslash\tilW/W_{\bP}}\wt\Gamma_{0} \tilw \bP_\infty.
\end{equation*}

Recall the apartment $\frA$ and the barycenter $\xi$ of the facet corresponding to $\bP_{\infty}$ as in \S\ref{ss:bij}. The set $\Psi_{\aff}$ of affine roots of $G(F_{\infty})$ are certain affine functions on $\frA$. For a subgroup $\bK$ of $G(F_{\infty})$, we use $\Psi(\bK)$ to denote the set of affine roots whose root subgroups are contained in $\bK$.  Therefore $\Psi(\bP_\infty)=\{\alpha\in\Psi_{\aff}|\alpha(\xi)\geq0\}$, $\Psi(\Gamma_{0})=\{\alpha\in\Psi_{\aff}|\alpha(\xi)\leq0\}$, and $\Psi(L_{\bP})=\{\alpha\in\Psi_{\aff}|\alpha(\xi)=0\}$ 

By the definition of $m=m(\bP)$, we have $\alpha(\xi)\in\frac{1}{m}\ZZ$. Moreover, since $\xi$ is the barycenter of its facet, $\Psi(\bP^{+}_\infty)-\Psi(\bP^{++}_\infty)=\{\alpha\in\Psi_{\aff}|\alpha(\xi)=\frac{1}{m}\}$ (these are exactly the affine roots that appear in $V_{\bP}$), and $\Psi(\bP^{++}_\infty)=\{\alpha\in\Psi_{\aff}|\alpha(\xi)\geq\frac{2}{m}\}$.

Suppose $f$ is nonzero on the double coset $\wt\Gamma_{0}\tilw \ell\bP^{+}_\infty$ for some $\tilw\in\tilW$ and $\ell\in L_{\bP}$.  Note that $f$ is a function on the quotient $\Gss\backslash(\wt\Gamma_{0}\tilw \ell\bP^{+}_\infty)/\bP^{++}_{\infty}$ (where $\Gss:=\ker(\wt\Gamma_{0}\to\tLab_{\bP})$),  and is $(V_{\bP}, \psi\phi)$-eigen. For $\alpha\in\Psi(\bP^{+}_\infty)-\Psi(\bP^{++}_\infty)$ (equivalently, $\alpha(\xi)=\frac{1}{m}$), we identify the affine root subgroup of $\alpha$ with the $\alpha$-weight space $V_{\bP}(\alpha)$ of $V_{\bP}$. Suppose also $\tilw\alpha\in\Psi(\Gss)$ (equivalently $(\tilw\alpha)(\xi)\leq0$). then for any $u_{\alpha}\in V_{\bP}(\alpha)$, we have
\begin{equation*}
\psi(\jiao{\phi, \Ad(\ell^{-1})u_{\alpha}})f(\tilw \ell)=f(\tilw \ell \Ad(\ell^{-1})u_{\alpha})=f(\tilw u_{\alpha}\ell)=f(\Ad(\tilw)u_{\alpha} \tilw\ell)=f(\tilw\ell).
\end{equation*}
Therefore we must have $\jiao{\phi,\Ad(\ell^{-1})V_{\bP}(\alpha)}=0$. The annihilator of $\bigoplus_{\alpha(\xi)=\frac{1}{m}, (\tilw\alpha)(\xi)\leq0}V_{\bP}(\alpha)\subset V_{\bP}$ in the dual space is $V^{*,\tilw}_{\bP}:=\bigoplus_{\alpha(\xi)=\frac{1}{m}, (\tilw\alpha)(\xi)>0}V_{\bP}(\alpha)^{*}=\bigoplus_{\beta(\xi)=-\frac{1}{m}, (\tilw\beta)(\xi)<0}V^{*}_{\bP}(\alpha)$. Then the above discussion shows that the orbit $L_{\bP}\cdot\phi$ intersects $V^{*,\tilw}_{\bP}$.

For $\beta(\xi)=\frac{1}{m}$, the condition $(\tilw\beta)(\xi)<0$ is equivalent to saying that $(\tilw\alpha)(\xi)\leq\alpha(\xi)$, or $\jiao{\bar{\beta},\tilw^{-1}\xi-\xi}\leq0$, where $\bar{\beta}$ is the vector part of $\beta$. In other words, in the weight decomposition of $V^{*,\tilw}_{\bP}$ under the torus $\SS$, a weight $\bar{\beta}\in\xch(\SS)$ appears if and only if it satisfies $\jiao{\bar{\beta},\tilw^{-1}\xi-\xi}\leq0$.

Suppose $\tilw^{-1}\xi\neq \xi$, let $\lambda=m(\tilw^{-1}\xi-\xi)\in\xcoch(\SS)$. Any point of $V^{*,\tilw}_{\bP}=\oplus_{\jiao{\bar{\beta},\lambda}\leq0} V^{*}_{\bP}(\beta)$ under the action of the one-dimensional torus $\lambda(\GG_m)$ has a limit point in $\oplus_{\jiao{\bar{\beta},\lambda}=0}V^{*}_{\bP}(\beta)$. Since the orbit $L_{\bP}\cdot\phi$ is closed, it contains a point $\phi'\in\oplus_{\jiao{\bar{\beta},\lambda}=0}V^{*}_{\bP}(\beta)$. But then the torus $\lambda(\GG_m)$ fixes $\phi'$, contradicting the assumption that $\phi$, hence $\phi'$, should have finite stabilizer under $L_{\bP}$.  This shows that $f$ must be zero on $\Gamma\tilw \bP_\infty$ for $\tilw^{-1}\xi\neq \xi$. 

Those $\tilw$ with $\tilw^{-1}\xi=\xi$ are precisely those in $\tilW_{\bP}$, therefore, $f$, as a function on $ G(F_\infty)$, is supported on the unit coset $\wt\Gamma_{0}\bP_{\infty}=\wt\Gamma_{0}\times\bP^{+}_\infty$. The $(\wt\bP_{0},\chi)$-eigen property of $f$ as a function on $G(\AA_{F})$  implies that it is left $(\wt\Gamma_{0},\chi)$-eigen as a function on $G(F_{\infty})$. Together with the right $(\bP^{+}_\infty,\psi\phi)$-eigen property, the function $f$ is unique up to a scalar. We have shown that $\dim \calA=1$.

Next we check that $\calA$ consists of cuspidal functions. In fact, for any $f\in\calA$, $g\in G(\AA_{F})$ and any proper $F$-parabolic $P\subset G$ with unipotent radical $U_{P}$, the constant term $\int_{U_{P}(F)\backslash U_{P}(\AA_{F})}f(ng)dn$ can be written as a finite sum of the form $h(g_{\infty}):=\int_{\Gamma\backslash U_{P}(F_{\infty})}f(g^{\infty}, n_{\infty}g_{\infty})dn_{\infty}$, for some discrete subgroup $\Gamma\subset U_{P}(F_{\infty})$ and $g=(g^{\infty}, g_{\infty})$ where $g_{\infty}\in G(F_{\infty}), g^{\infty}\in G(\AA^{\infty}_{F})$. The function $h:G(F_{\infty})\to \Ql$ defined above is left $U_{P}(F_{\infty})$-invariant and right $(\bP^{+}_{\infty}, \psi\phi)$-eigen, hence induces a $U_{P}(F_{\infty})$-invariant map $c_{h}:\textup{c-Ind}^{G(F_{\infty})}_{\bP^{+}_{\infty}}(\bar{\psi}\phi)\to\Ql$ by $h'\mapsto \int_{G(F_{\infty})/\bP^{+}_{\infty}}h(x)h'(x^{-1})dx$. However, as we discussed in \S\ref{ss:epi}, $\textup{c-Ind}^{G(F_{\infty})}_{\bP^{+}_{\infty}}(\bar{\psi}\phi)$ is a finite direct sum of supercuspidal representations, hence has zero Jacquet module with respect to $P$. Therefore $c_{h}=0$ and $h$ must be the zero function. This implies that the constant term $\int_{U_{P}(F)\backslash U_{P}(\AA_{F})}f(ng)dn$ must be zero.

Since cuspidal functions belong to the discrete spectrum, we have
\begin{equation}\label{mult}
\dim\calA=\sum_{\pi}m(\pi)\dim\pi_{0}^{(\wt\bP_{0},\chi)}\dim\pi_{\infty}^{(\bP_{\infty}^{+},\psi\phi)}.
\end{equation}
where the sum is over isomorphism classes of $\pi$ satisfying the conditions in \S\ref{ss:auto}, and $m(\pi)$ is the multiplicity that $\pi$ appears in the automorphic spectrum of $G(\AA_{F})$. Since $\dim\calA=1$, the right side of \eqref{mult} only has one nonzero term, in which all factors are equal to one.
\end{proof}

\section{Generalized Kloosterman sheaves}\label{s:cons} We keep the notations from \S\ref{ss:auto}. In this section, we will construct generalized Kloosterman sheaves as Galois representations attached to the automorphic representations $\pi(\chi,\phi)$ in Proposition \ref{p:uniquefunction}. The construction uses ideas from the geometric Langlands correspondence.

\subsection{A sheaf on the moduli space of $G$-bundles} 
Let $\Bun_{G}(\Pss,\bP^{++}_\infty)$ be the moduli stack of $G$-bundles over $X=\PP^{1}$ with $\Pss$-level structures at $0$ and $\bP^{++}_\infty$-level structures at $\infty$. This is an algebraic stack over $k$. For the construction of moduli stacks with parahoric level structures, such as $\Bun_{G}(\bP_{0}, \bP_{\infty})$, see \cite[\S4.2]{GS}; the moduli stack $\Bun_{G}(\Pss,\bP^{++}_\infty)$ is an $\Lab_{\bP}\times (\bP_{\infty}/\bP_{\infty}^{++})$-torsor over  $\Bun_{G}(\bP_{0}, \bP_{\infty})$. In the sequel we abbreviate $\Bun:=\Bun_{G}(\Pss,\bP^{++}_{\infty})$. The trivial $G$-bundle with the tautological level structures at $0$ and $\infty$ gives a base point $\star\in\Bun(k)$.

There is an action of $\tLab_{\bP}\times V_{\bP}$ on $\Bun$ because $\wt\bP_{0}$ normalizes $\Pss$ and $\bP_{\infty}^{+}$ normalizes $\bP^{++}_{\infty}$. The character $\chi$ defines a rank one Kummer local system $\calL_\chi$ on $\tLab_{\bP}$ such that its Frobenius traces at $k$-rational points give back the character $\chi:\tLab_{\bP}(k)\to \Ql^{\times}$. Let $S$ be a scheme and $s:S\to V^{*}_{\bP}$ be a morphism. We also consider $V_{\bP}\times S$ as a constant additive group over $S$, over which we have a rank one local system $\AS_{S}:=(\id\times s)^{*}\jiao{,}^{*}\AS_{\psi}$ where $\jiao{,}: V_{\bP}\times V^{*}_{\bP}\to\GG_a$ is the natural pairing and $\AS_\psi$ is the rank one Artin-Schreier local system on $\GG_{a}$ given by the additive character $\psi$. 

Both $\calL_{\chi}$ and $\AS_{S}$ are character sheaves. We spell out what this means for $\AS_{S}$. Let $a: V_{\bP}\times V_{\bP}\times S\to V_{\bP}\times S$ be the addition map in the first two variables, and let $p_{23}:V_{\bP}\times V_{\bP}\times S\to V_{\bP}\times S$ be the projection onto the last two factors. Then there is an isomorphism $a^{*}\AS_{S}\cong p_{23}^{*}\AS_{S}$, which satisfies the usual cocycle relation when pulled back to $V^{3}_{\bP}\times S$.

For any scheme $S$ over $V^{*}_{\bP}$, we consider the product $\Bun\times S$, on which the algebraic group $\tLab_{\bP}$ and the constant group scheme $V_{\bP}\times S$ acts. We may then talk about the derived category of $(\tLab_{\bP},\calL_{\chi})$-equivariant and $(V_{\bP}\times S, \AS_{S})$-equivariant $\Ql$-complex on $\Bun\times S$. We denote this category by $\calD(\chi, S)$.

Let $\Bun_{G}(\wt\bP_{0}, \bP^{+}_{\infty})$ be the quotient of $\Bun$ by the $\tLab_{\bP}\times V_{\bP}$-action. The parahoric analog of the Birkhoff decomposition \cite[Proposition 1.1]{HNY} gives a decomposition
\begin{equation}\label{Birk}
\Bun_{G}(\wt\bP_{0}, \bP^{+}_{\infty})(k)\cong\wt\Gamma_{0}\backslash G(F_{\infty})/\bP^{+}_{\infty}=\bigsqcup_{w\in \tilW_{\bP}\backslash \tilW/W_{\bP}}\wt\Gamma_{0}\backslash(\wt\Gamma_{0}\tilw\bP_{\infty})/\bP^{+}_{\infty}.
\end{equation}
Moreover, the same is true when $k$ is replaced by its algebraic closure (and with $G(F_{\infty}), \bP_{\infty}$ etc. base changed to $\bark$). By the parahoric analog of \cite[Corollary 1.3(3)]{HNY}, when $\tilw=1$, the unit double coset is an open point with trivial stabilizer, which is also the image of $\star$. The preimage $U$ of this open point in $\Bun$ is then a $\tLab_{\bP}\times V_{\bP}$-torsor trivialized by the base point $\star$, giving an open immersion 
\begin{equation*}
j:U\cong\tLab_{\bP}\times V_{\bP}\hookrightarrow\Bun.
\end{equation*}
For $?=!$ or $*$, let
\begin{equation*}
A_{?}(\chi, S)=(j\times\id_{S})_?(\calL_{\chi}\boxtimes\AS_{S})\in \calD(\chi,S).
\end{equation*}

\begin{lemma}\label{l:clean}  Let $S$ be a scheme over $V^{*,\st}_{\bP}$. 
\begin{enumerate}
\item\label{Aclean} Any object  $A\in \calD(\chi, S)$ has vanishing stalks outside $U\times S$. In particular, the canonical map $A_{!}(\chi, S)\to A_{*}(\chi, S)$ is an isomorphism. We shall denote them by $A(\chi,S)$.
\item The functor
\begin{eqnarray*}
D^{b}_{c}(S)&\to& \calD(\chi,S)\\
C&\mapsto& (j\times\id_{S})_{!}((\calL_{\chi}\boxtimes\AS_{S})\otimes\pr_{S}^{*}C)
\end{eqnarray*}
is an equivalence of categories. Here $\pr_{S}: U\times S\to S$ is the projection.
\end{enumerate}
\end{lemma}
\begin{proof} (2) is an immediate consequences of (1). To show (1), it suffices to treat the case $S=\Spec\ \bark$, which corresponds to a point $\phi\in V^{*,\st}(\bark)$. We base change the situation to $\bark$ without changing notation. By the Birkhoff decomposition \eqref{Birk}, we need to show that any sheaf on $\Gss\backslash(\wt\Gamma_{0}\tilw\bP_{\infty})/\bP^{++}_{\infty}$ that are $(V_{\bP}, \phi^{*}\AS_{\psi})$-equivariant on the right must be zero. The argument is a straight-forward sheaf-theoretic analog of the argument given in Proposition \ref{p:uniquefunction}. 
\end{proof}

\subsection{Hecke operators}
For more details on the Satake equivalence and Hecke operators, we refer to \cite[\S2.3-2.4]{HNY}. In particular, our Satake category consists of weight zero complexes as normalized in \cite[Remark 2.10]{HNY} (note this involves the choice of a half Tate twist). Here we only set up the notation in order to state our main result.

The Hecke correspondence for $\Bun=\Bun(\Pss,\bP^{++}_{\infty})$ classifies $(x, \calE,\calE',\tau)$ where $x\in \tpline$, $\calE,\calE'\in\Bun$ and $\tau: \calE|_{\tX-\{x\}}\isom\calE'|_{\tX-\{x\}}$ is an isomorphism of $G$-torsors preserving the level structures at $0$ and $\infty$. Let $\pi:\Hk\to\tpline$ be the morphism that remembers only $x$, and let $\Hk_{x}$ be the fiber of $\pi$ over $x\in\tpline$. For fixed $x\in\tpline$, there is an evaluation map $\ev_{x}: \Hk_{x}\to [L^{+}_{x}\GG\backslash L_{x}\GG/L^{+}_{x}\GG]$ by recording the behavior of $\tau$ around $x$. For every $V\in\Rep(\dualG)$, the geometric Satake correspondence gives a perverse sheaf $\IC_{V}$ (normalized to be pure of weight zero) on the affine Grassmannian $\Gr_{x}=L_{x}\GG/L^{+}_{x}\GG$ of $\GG$ that descends to the stack $[L^{+}_{x}\GG\backslash L_{x}\GG/L^{+}_{x}\GG]$. The pullback $\ev_{x}^{*}\IC_{V}$ defines a complex on $\Hk_{x}$. When $x$ varies in $\tpline$, it can be shown that the complexes $\ev_{x}^{*}\IC_{V}$ can be glued to a complex $\IC^{\Hk}_{V}$ on $\Hk$, whose restriction to each fiber $\Hk_{x}$ is $\ev_{x}^{*}\IC_{V}$.

Let $S$ be a scheme of finite type over $k$. Consider the diagram
\begin{equation}\label{Hk}
\xymatrix{ & S\times \Hk\ar[dl]_{\id_{S}\times\oll{h}}\ar[dr]^{\id_{S}\times\orr{h}}\ar[rr]^{\pi} & & \tpline\\
S\times \Bun & & S\times \Bun}
\end{equation}
where $\oll{h}$ and $\orr{h}$ send $(x, \calE,\calE',\tau)$ to $\calE$ and $\calE'$  respectively. For each $V\in\Rep(\dualG)$, we define the geometric Hecke operator (relative to the base $S$) to be the functor
\begin{eqnarray}
\notag T^{V}_{S}&:&D^{b}_{c}(S\times \Bun)\to D^{b}_{c}(\tpline\times S\times \Bun)\\
\label{TV}&& A\mapsto (\pi\times\id_{S}\times\orr{h})_{!}\left((\id_{S}\times\oll{h})^{*}A\otimes\IC^{\Hk}_{V}\right).
\end{eqnarray}

\begin{defn} 
An {\em $S$-family of Hecke eigensheaves} is the data $(A, E_{\dG}, \{\iota_{V}\}_{V\in\Rep(\dualG)})$ where
\begin{itemize}
\item $A\in D^{b}(S\times \Bun)$;
\item $E_{\dG}$ is a $\dG$-local system on $\pline\times S$. We denote the pullback of $E_{\dG}$ to $\tpline\times S$ by $E_{\dualG}$, which is a $\dualG$-local system with a $\mu_{e}$-equivariant structure (with $\mu_{e}$ acting both on $\dualG$ via pinned automorphisms and on $\tpline$).
\item For each $V\in\Rep(\dualG)$, $\iota_{V}$ is an isomorphism over $\tpline\times S\times \Bun$
\begin{equation*}
\iota_{V}: T^{V}_{S}(A)\cong E^{V}_{\dualG}\otimes_{S}A,
\end{equation*}
where $E^{V}_{\dualG}$ is the local system associated to $E_{\dualG}$ and the representation $V$ of $\dualG$. Here $E^{V}_{\dualG}\otimes_{S}A$ means the pullback of $E^{V}_{\dualG}\boxtimes A$ along the diagonal map $\id_{\tX}\times\Delta_{S}\times\id_{\Bun}: \tpline\times S\times \Bun\to \tpline\times S\times S\times \Bun$. 
\end{itemize}
These data should satisfy the following conditions.
\begin{enumerate}
\item The isomorphisms $\iota_{V}$ should be compatible with the tensor structure of $\Rep(\dualG)$. For details, see \cite[p.163-164]{Gait}.
\item The isomorphisms $\iota_{V}$ should be compatible with the actions of $\mu_{e}$. In other words, for $V\in\Rep(\dualG)$, let $V^{\sigma}$ be the same vector space on which $\dualG$ acts via $\dualG\xrightarrow{\sigma}\dualG\to\GL(V)$. Let $\sigma_{\tpline}:\tpline\to\tpline$ be the action of $\mu_{e}$ on $\tpline$. Then we have a canonical isomorphism of functors $T^{V^{\sigma}}_{S}(-)\cong(\sigma_{\tpline}\times\id_{S\times\Bun})^{*}T^{V}_{S}(-)$ (coming from the $\mu_{e}$-equivariant nature of the Hecke correspondence) and an isomorphism of local systems $E^{V^{\sigma}}_{\dualG}\cong(\sigma_{\tpline}\times\id_{S})^{*}E^{V}_{\dualG}$ since $E_{\dualG}$ comes from the $\dG$-local system $E_{\dG}$. Then the following diagram is required to be commutative
\begin{equation*}
\xymatrix{T^{V^{\sigma}}_{S}(A)\ar[rr]^{\iota_{V^{\sigma}}}\ar[d]^{\wr} & & E^{V^{\sigma}}_{\dualG}\otimes_{S} A\ar[d]^{\wr}\\
(\sigma_{\tpline}\times\id_{S\times\Bun})^{*}T^{V}_{S}(A)\ar[rr]^{(\sigma_{\tpline}\times\id_{S})^{*}\iota_{V}} && (\sigma_{\tpline}\times\id_{S})^{*}E^{V}_{\dualG}\otimes_{S} A}
\end{equation*}
For details, see \cite[\S2.4]{HNY}.
\end{enumerate}
\end{defn}

\begin{theorem}\label{th:main} Let $S$ be a scheme over $V^{*,\st}_{\bP}$. The complex
$A(\chi, S)$ can be given the structure of an $S$-family of Hecke eigensheaves. We denote the corresponding eigen $\dG$-local system by $\Kl_{\dG, \bP}(\chi, S)$, which is a $\dG$-local system over $\pline\times S$.
\end{theorem}
\begin{proof} We abbreviate $A(\chi, S)$ by $A$. 
We first prove that $T^{V}_{S}(A)\cong E^{V}_{S}\otimes_{S} A$  for some {\em complex} $E^{V}_{S}$ on $\pline\times S$. Note that the Hecke operators preserve the $(\tLab_{\bP},\chi)$ and $(V_{\bP}\times S,\AS_{S})$-equivariant structures, hence they give functors
\begin{equation*}
T^{V}_{S}:\calD(\chi,S)\to \calD(\chi, \tpline\times S).
\end{equation*}
Applying  Lemma \ref{l:clean}(2) to $S'=\tpline\times S$, objects on the right side above are of the form $(\id_{\tpline}\times S\times j)_{!}(\calL_{\chi}\boxtimes\AS_{S}\otimes\pr^{*}_{S}E^{V}_{S})\cong E^{V}_{S}\otimes_{S}A$ for some $E^{V}_{S}\in D^{b}_{c}(\tpline\times S)$.

By the construction of the Hecke operators and its compatibility with the geometric Satake correspondence, the assignment $V\mapsto E^{V}_{S}$ carries associativity and commutativity constraints for tensor functors (with respect to the sheaf-theoretic tensor product of $E^{V}_{S}$ and $E^{V'}_{S}$). 

It remains to show that $E^{V}_{S}$ is a local system. We first consider the case where $S$ is a geometric point. The argument in this case is the same as in the Iwahori level case treated in \cite[\S4.1-4.2]{HNY}: one first argues that $E^{V}_{S}[1]$ is a perverse sheaf, and then use the tensor structure of $V\mapsto E^{V}_{S}$ to show that it is indeed a local system. Let us only mention one key point in proving that $E^{V}[1]$ is perverse, that is to show an analogue of \cite[Remark 4.2]{HNY}: the map $\pi\times\orr{h}: \oll{h}^{-1}(U)\subset \Hk\to \tpline\times\Bun$ is affine. For this we only need to argue that for any $x\in \tpline(R)$ where $R$ is a finitely generated $k$-algebra, the preimage of $U$ in the affine Grassmannian $\Gr_{x}$ over $R$ is affine. Note that $U$ is the non-vanishing locus of some line bundle $\calL$ on $\Bun$. The pullback of $\calL$ to $\Gr_{x}$ has to be ample relative to the base $\Spec R$ because the relative Picard group of $\Gr_{x}$ is $\ZZ$ by \cite[Corollary 12]{Falt}. Therefore the preimage of $U$ in $\Gr_{x}$ is affine. 

We then treat the general case. Since the Hecke operators $T^{V}_{S}$ commute with pullback along a base change map $S'\to S$, so does the formation of $E^{V}_{S}$. Therefore, for every geometric point $s\in S$, $E^{V}_{S}|_{\tpline\times\{s\}}$ is isomorphic to $E^{V}_{s}$, which is a local system on $\tpline\times\{s\}$ by the argument of the previous paragraph. Let $t\in S$ be a geometric point that specializes to another geometric point $s$, then we have the specialization map $\textup{sp}^{V}_{s\to t}:E^{V}_{s}\to E^{V}_{t}$. Note that $E^{V}_{s}$ and $E^{V}_{t}$ are plain vector spaces of dimension $\dim V$. As $V\in\Rep(\dualG)$ varies, the functors $V\mapsto E^{V}_{s}$ and $V\mapsto E^{V}_{t}$ are fiber functors of the rigid tensor category $\Rep(\dualG)$ and the maps $\{\textup{sp}^{V}_{s\to t}\}_{V}$ gives a morphism of tensor functors. Therefore each $\textup{sp}^{V}_{s\to t}$ is an isomorphism by \cite[Proposition 1.13]{DM}. This being true for every pair of specialization $(t,s)$, $E^{V}_{S}$ is a local system over $S$. The proof is complete.
\end{proof}

Let $\grot$ be the one-dimensional torus acting simply transitively on $\tpline$. It also acts on $\pline$ via $e$th power. It also acts on every standard parahoric subgroup of the loop groups $L_{0}G$ and $L_{\infty}G$, hence on $L_{\bP}$ and $V_{\bP}$. We denote action of $\l\in\grot$ on various spaces by $\l\cdot_{\rot}(-)$.

Viewing $L_{\bP}$ as the Levi factor of $\bP_{\infty}$, there is an action of $L_{\bP}\rtimes\grot$ on $\Bun$ where $L_{\bP}$ changes the $\bP^{++}_{\infty}$-level structures. When $S=V^{*,\st}_{\bP}$, this action can be extended to every space in the Hecke correspondence diagram \eqref{Hk} by making $L_{\bP}\rtimes\grot$ act on $S=V^{*,\st}_{\bP}$ in the natural way, such that all maps in \eqref{Hk} are $L_{\bP}\rtimes\grot$-equivariant.  Consequently, we have

\begin{lemma}\label{l:Kl des}
The $\dG$-local system $\tKl_{\dG, m}(\chi, V^{*,\st}_{\bP})$ over $\pline\times V^{*,\st}_{\bP}$ is equivariant under the action of $L_{\bP}\rtimes\grot$ on $\pline\times V^{*,\st}_{\bP}$. Here $\grot$ acts diagonally on $\pline\times V^{*,\st}_{\bP}$. Therefore, $\Kl_{\dG,\bP}(\chi, V^{*,\st}_{\bP})$ descends to an $L_{\bP}$-equivariant $\dualG$-local system  over $V^{*,\st}_{\bP}$, equipped with a $\mu_{e}$-equivariant structure (which acts on both $\dualG$ and on $V^{*,\st}_{\bP}$ via the embedding $\mu_{e}\incl\grot$). We denote this $\dualG$-local system on $V^{*,\st}_{\bP}$ by $\Kl_{\dualG, \bP}(\chi)$.
\end{lemma}

For a stable linear functional $\phi:V_{\bP}\to k$, viewed as a point $\phi\in V^{*,\st}_{\bP}(k)$, the notation $\Kl_{\dualG, \bP}(\chi,\phi)$ and $\Kl_{\dG,\bP}(\chi,\phi)$ is defined as in Theorem \ref{th:main} by taking $S$ to be the $\Spec\ k\xrightarrow{\phi}V^{*,\st}_{\bP}$.

A direct consequence of Theorem \ref{th:main} and Lemma \ref{l:Kl des} is
\begin{cor}\label{c:Klphi} For a stable linear functional $\phi:V_{\bP}\to k$, the $\dG$-local system $\Kl_{\dG,\bP}(\chi,\phi)$ over $\pline$ is the global Langlands parameter attached to the automorphic representation $\pi(\chi,\phi)$ in Proposition \ref{p:uniquefunction}. Moreover, $\Kl_{\dualG, \bP}(\chi,\phi)$ is isomorphic to the pullback of $\Kl_{\dualG, \bP}(\chi)$ along the map $a_{\phi}: \tpline\cong\grot\to V^{*,\st}_{\bP}$ given by $\l\mapsto \l\cdot_{\rot}\phi$.
\end{cor}

The case considered in \cite{HNY} is $\bP=\bI$. In this case, $L_{\bI}=\SS$ and $V_{\bI}=\oplus_{i=0}^{r} V_{\bI}(\alpha_i)$ for affine simple roots $\{\alpha_i\}$ of $G$. If $G$ is adjoint, we may identify the quotient $[V^{*, \st}_{\bI}/\SS]$ with $\GG_m$ (via projection to the factor $V_{\bP}(\alpha_{0})^{*}$). So in this case the Kloosterman sheaf $\Kl_{\dualG,\bI}(\chi)$ descends to a $\dualG$-local system on $\GG_m$, without any dependence on the functional $\phi$.

\subsection{Calculation of the local system}\label{ss:cal}

The moduli stack $\Bun_{G}(\wt\bP_{0},\bP_{\infty}^{+})$ has a unique point which has trivial automorphism group. We denote the corresponding $G$-bundle over $X$ with level structure by $\calE$. For $x\in\pline$, let $\frG$ be the automorphisms of $\calE|_{X-\{1\}}$ preserving the level structures at $0$ and $\infty$. This is a group ind-scheme over $k$. Now consider $1\in\tX$, a preimage of  $1\in X$. To distinguish them we denote the $1\in\tX$ by $\wt{1}$. Evaluating an automorphism $g\in\frG$ in the formal neighborhood of $\wt{1}$ (note $\calE$ gives a $\GG$-bundle over $\tX$) gives a morphism
\begin{equation*}
\ev_{\wt{1}}: \frG\to L^{+}_{\wt{1}}\GG\backslash L_{\wt{1}}\GG/L^{+}_{\wt{1}}\GG.
\end{equation*}
We shall abbreviate $L_{\wt{1}}\GG$ and $L^{+}_{\wt{1}}\GG$ by $L\GG$ and $L^{+}\GG$. 

Let $\lambda\in\xcoch(\TT)$ be a dominant coweight of $\GG$, which defines an affine Schubert variety $\Gr_{\leq\l}$ in the affine Grassmannian $\Gr=L\GG/L^{+}\GG$. Let $\frG_{\leq\l}$ be the preimage of $L^{+}\GG\backslash \Gr_{\leq\l}\subset L^{+}\GG\backslash L\GG/L^{+}\GG$ under $\ev_{\wt{1}}$. The intersection complex $\IC_{\l}$ of $\frG_{\leq\l}$ corresponds, under the geometric Satake equivalence, to the irreducible representation $V_{\l}$ of $\dualG$ with highest weight $\l$.

We have two more evaluation maps, at $0$ and $\infty$
\begin{equation*}
(\ev_{0}, \ev_{\infty}): \frG_{\leq\l}\to \wt{\bP}_{0}\times\bP^{+}_{\infty}.
\end{equation*}
Composing with the projections $\wt{\bP}_{0}\to \tLab_{\bP}$ and $\bP^{+}_{\infty}\to V_{\bP}$, we get\begin{equation}\label{ff}
(f',f''):\frG_{\leq\l}\to \tLab_{\bP}\times V_{\bP}.
\end{equation}

\begin{prop}\label{p:Four} Let $\Four_{\psi}:D^{b}_{c}(V_{\bP})\to D^{b}_{c}(V^{*}_{\bP})$ be the Fourier-Deligne transform (without cohomological shift). We have
\begin{equation*}
\Kl^{V_\lambda}_{\dualG,\bP}(\chi)\cong \Four_{\psi}\left(f''_!(f'^*\calL_\chi\otimes\ev^{*}_{\wt{1}}\IC_\lambda)\right)|_{V_{\bP}^{*,\st}}.
\end{equation*}
\end{prop}
\begin{proof}
We shall work with the base $S=V^{*,\st}_{\bP}$. 
By construction, $\Kl^{V_{\l}}_{\dualG,\bP}(\chi)$ is the restriction of $T^{V_{\l}}_{S}(A(\chi, S))$ to the point $\{\wt{1}\}\times S\times\{\star\}\subset \tpline\times S\times \Bun$. Let $\Hk_{1}\subset\Hk$ be the preimage of $(\wt{1},\star)$ under $(\pi, \orr{h})$, and $\Hk^{U}_{1}\subset\Hk_{1}$ be the preimage of $U$ under $\oll{h}_{1}$ (restriction of $\oll{h}$). By proper base change and the definition of the Hecke operators \eqref{TV}, we have
\begin{equation}\label{KlS}
\Kl^{V_{\l}}_{\dualG,\bP}(\chi)=\pr_{S,!}((\id_{S}\times\oll{h}_{1})^{*}\AS_{S}\otimes\IC^{\Hk_{1}}_{\l}).
\end{equation}
where
\begin{equation*}
S\times U\xleftarrow{\id_{S}\times\oll{h}_{1}}S\times\Hk^{U}_{1}\xrightarrow{\pr_{S}}S
\end{equation*}
and $\IC^{\Hk_{1}}_{\l}$ is the restriction of $\IC^{\Hk}_{\l}$ to $\Hk_{1}$. By the definition of $\Hk$,  $\Hk^{U}_{1}$ classifies isomorphisms pairs $(\calE',\tau)$ where $\calE'\in U$ and $\tau$ is an isomorphism between $\star|_{\tpline-\{\wt{1}\}}$ and $\calE'|_{\tpline-\{\wt{1}\}}$ preserving the $\Pss$ and $\bP^{++}_{\infty}$-level structures. Note that both $\star$ and $\calE'\in U$ maps to the open point $\calE$ in $\Bun_{G}(\wt\bP_{0}, \bP^{+}_{\infty})$, therefore we have an isomorphism $\iota:\frG\cong\Hk^{U}_{1}$ sending an automorphism $\alpha$ of $\calE|_{\tpline-\{\wt{1}\}}$ (viewed as an automorphism of $\star|_{\tpline-\{\wt{1}\}}$ not necessarily preserving the $\Pss$ and $\bP^{++}_{\infty}$-level structures) to the pair $(\calE'_{\alpha}, \alpha)\in\Hk^{U}_{1}$ where $\calE'_{\alpha}$ is the unique point in $U$ obtained by modifying the $\Pss$ and $\bP^{++}_{\infty}$-level structures of $\star$ so that $\alpha$ becomes an isomorphism between $\star|_{\tpline-\{\wt{1}\}}$ and $\calE'_{\alpha}|_{\tpline-\{\wt{1}\}}$ preserving the $\Pss$ and $\bP^{++}_{\infty}$-level structures. Moreover, under $\iota$, $\oll{h}_{1}:\Hk^{U}_{1}\to U$ is identified with $(f',f'')$ defined in \eqref{ff}. Therefore, \eqref{KlS} implies
\begin{eqnarray*}
\Kl^{V_{\l}}_{\dualG,\bP}(\chi)&\cong&\pr_{S,!}(f'^{*}\calL_{\chi}\otimes(\id_{S}\times f'')^{*}\AS_{S}\otimes\ev^{*}_{\wt{1}}\IC_{\l})\\
&=&\pr_{V^{*}_{\bP},!}(f'^{*}\calL_{\chi}\otimes(\id_{V^{*}_{\bP}}\times f'')^{*}\jiao{\cdot,\cdot}^{*}\AS_{\psi}\otimes\ev^{*}_{\wt{1}}\IC_{\l})|_{S}
\end{eqnarray*}
Using proper base change, the above is further isomorphic to
\begin{eqnarray*}
\pr_{V^{*}_{\bP},!}\left((\id_{V^{*}_{\bP}}\times f'')_{!}(f'^{*}\calL_{\chi}\otimes\ev^{*}_{\wt{1}}\IC_{\l})\otimes\jiao{\cdot,\cdot}^{*}\AS_{\psi}\right)|_{S}=\Four_{\psi}\left(f''_{!}(f'^{*}\calL_{\chi}\otimes\ev^{*}_{\wt{1}}\IC_{\l})\right)|_{S}.
\end{eqnarray*}
\end{proof}

\section{Unipotent monodromy}
The goal of this section is to study the monodromy of the Kloosterman sheaves $\Kl_{\dualG,\bP}(\chi,\phi)$ at $0$, especially when $\chi=1$.

\subsection{Lusztig's theory of (two-sided) cells}

Let $\Wa$ be the affine Weyl group attached to $G(K)$. There is a partition of $\Wa$ into finitely many {\em two-sided cells}. There is a partial order among these cells. The largest cell is the singleton $\{1\}$; the smallest one contains the longest element $w_0$ in the Weyl group $W_{\bQ}$ of a special parahoric subgroup $\bQ$. Each cell $\unc$ is assigned an integer $a(\unc)$ satisfying $0\leq a(\unc)\leq \min\{\ell(w)|w\in\unc\}$.

Let $\dualG^{\sigma,\circ}$ be the neutral component of the fixed point of $\sigma$ on $\dualG$.
 
\begin{theorem}[Lusztig {\cite[IV, Theorem 4.8]{cells}}] \label{th:Lu}
There is an order preserving bijection
\begin{equation*}
\{\textup{two-sided cells in }\Wa\}\longleftrightarrow\{\textup{unipotent conjugacy classes in }\dualG^{\sigma,\circ}\}
\end{equation*}
such that if $\unc\leftrightarrow\unu$, then $a(\unc)=\dim\calB_u$ where $\calB_u$ is the Springer fiber of $u\in\unu$ (inside the flag variety of $\dualG^{\sigma,\circ}$).
\end{theorem}

\begin{remark}\label{r:cell qsplit} Lusztig's paper only dealt with the case of split $G$. However, for our quasi-split $G$, the affine Weyl group $\Wa$ is isomorphic to the affine Weyl group of a split group $G'$ as Coxeter groups. Moreover, one can choose $G'$ so that $\dualG'=\dualG^{\sigma,\circ}$. Therefore the quasi-split case of the above theorem follows from the split case.
\end{remark}

\begin{defn}\label{def:unu} For a standard parahoric subgroup $\bP\subset G(K)$, let  $\unc_{\bP}$ be the two-sided cell in $\Wa$ containing the longest element $w_{\bP}\in W_{\bP}$ (the Weyl group of $L_{\bP}$). We define $\unu_{\bP}$ to be the unipotent class of $\dualG^{\sigma,\circ}$ corresponding to $\unc_{\bP}$ under Theorem \ref{th:Lu}.
\end{defn}

\begin{lemma}\label{l:aw}For any standard parahoric $\bP$, we have $a(\unc_\bP)=\ell(w_{\bP})$. \end{lemma}
\begin{proof}
We may assume $G$ is simply-connected. Let $a=\ell(w_{\bP})$. Let $\bH:=C_{c}(\bI\backslash G(K)/\bI)$ be the Iwahori Hecke algebra and let $\{C_{w}\}_{w\in\Wa}$ be its Kazhdan-Lusztig basis. 
By definition, $a(\unc_{\bP})$ is equal to $a(w_{\bP})$, which is the largest power of $q^{1/2}$ appearing in the coefficient of $C_{w_{\bP}}$ in $C_{x}C_{y}$ for some $x,y\in\Wa$. Since $C_{w_{\bP}}^{2}=f(q^{1/2})C_{w_{\bP}}$ for $f(q^{1/2})=q^{-a/2}\sum_{x\in W_{\bP}}q^{\ell(x)}$ whose highest degree term is $q^{a/2}$, we have $a(w_{\bP})\geq a$. On the other hand, one always has $a(w_{\bP})\leq \ell(w_{\bP})=a$ by \cite[II, Proposition 1.2]{L}, therefore $a(\unc_{\bP})=a(w_{\bP})=a$. 
\end{proof}

We now give an alternative description of the unipotent class $\unu_{\bP}$. To state it, we recall the truncated induction defined by Lusztig \cite{LL} in a special case, which was first constructed by Macdonald \cite{M}. Let $W=\WW^{\sigma}$ be the $F$-Weyl group of $G$ (this is also $W_{\bQ}$ for a special parahoric $\bQ$, and is also the Weyl group of $\dualG^{\sigma,\circ}$), then $\fra=\xcoch(\TT)^{\sigma}_{\CC}$ is the reflection representation of $W$. The sign representation $\ep$ of $W_{\bP}$ appears in $\Res^{W}_{W_{\bP}}\Sym^{\ell(w_{\bP})}(\fra)$ with multiplicity one. The truncated induction $j^{W}_{W_{\bP}}(\ep)$ is the unique irreducible $W$-submodule of $\Sym^{\ell(w_{\bP})}(\fra)$ that contains the sign representation of $W_{\bP}$.

\begin{prop}\label{p:unip ind} Let $\bP$ be a standard parahoric subgroup of $G(K)$. \begin{enumerate}
\item Under the Springer correspondence, $j^{W}_{W_{\bP}}(\ep)$ corresponds to the unipotent class $\unu_{\bP}$ of $\dualG^{\sigma,\circ}$ together with the trivial local system on it.
\item Suppose $\bP\subset\bQ$ for some special parahoric $\bQ$, then $\bP$ corresponds to a standard parabolic subgroup $P\subset L_{\bQ}$, hence also to a standard parabolic subgroup $\widehat{P}\subset\dualG^{\sigma,\circ}$. Then $\unu_{\bP}$ is the Richardson class attached to $\widehat{P}$ (i.e., the unipotent class in $\dualG^{\sigma,\circ}$ that contains a dense open subset of the unipotent radical of $\widehat{P}$).
\end{enumerate}
\end{prop}
\begin{proof}
Remark \ref{r:cell qsplit} allows us to reduce to the case $G$ is split. 

(1) Let $\unu$ (together with the trivial local system on it) be the unipotent class in $\dualG$ that corresponds to $j_{W_{\bP}}^{W}(\ep)$ under the Springer correspondence. Springer theory tells that $\dim\calB_{u}=\ell(w_{\bP})$ for any $u\in\unu$.

Recall $\bH:=C_{c}(\bI\backslash G(K)/\bI)$ is the Iwahori Hecke algebra and let $\bH_{\bP}:=C(\bI\backslash \bP/\bI)$ be its parahoric subalgebra. Let $M$ be an irreducible $\bH$-module such that $M^{\bH_{\bP}}\neq0$ and that $C_{w}$ acts on $M$ by zero for any $w$ lying in cells $\unc\leq\unc_{\bP}$. Such an $\bH$-module can be found in the cell subquotient $\bH_{\unc_{\bP}}$. The Langlands parameter of $M$ is a triple $(u',s,\rho)$ where $u'$ is a unipotent class in $\dualG$. By the construction of Lusztig's bijection Theorem \ref{th:Lu}, $u'\in\unu_{\bP}$. Consider the $\bH$-module $\leftexp{*}{M}$ (changing the action of $T_{s}$ to $-qT^{-1}_{s}$ for simple reflections $s\in\Wa$, see \cite[IV \S1.8]{cells}), then $\leftexp{*}{M}$ contains a copy of the sign representation of $\bH_{\bP}$.
Applying \cite[IV, Lemma 7.2]{cells} to $\leftexp{*}{M}$, we conclude that $u'$ lies in the closure of $\unu$. Since we also have $\dim\calB_{u'}=a(\unc_{\bP})=\ell(w_{\bP})$ by Theorem \ref{th:Lu} and Lemma \ref{l:aw}, we must have $\unu=\unu_{\bP}$.

(2) follows from \cite[Theorem 3.5]{LS}.
\end{proof}

The main result of this section is
\begin{theorem}\label{th:unip} Assume $\chi=1$ and $G$ is split. Then for any stable functional $\phi:V_{\bP}\to k$, the local monodromy of the $\dualG$-local system $\Kl_{\dualG, \bP}(1,\phi)$ at $0$ is tame and is given by the conjugacy class $\unu_{\bP}$.
\end{theorem}

The proof of the theorem relies on deep results of Lusztig and Bezrukavnikov on cells in affine Weyl groups, and occupies \S\ref{ss:Hk0} to \S\ref{ss:pf unip}. 

If $\bP=\bI$, then $\unc_{\bP}$ is the maximal cell, hence $\unu_{\bP}$ is the regular unipotent class. This case of the theorem was proved in \cite[Theorem 1(2)]{HNY}. Another extremal case is when $\bP$ is a special parahoric of $G(K)$ (this happens only when $G$ is an odd unitary group and $m=2$), the quasi-split analog of the above theorem should say that $\Kl_{\dualG,\bP}(1,\phi)$ is unramified at $0$, which is the case as we will see in Proposition \ref{p:um2}(2).

\subsection{Tables of unipotent monodromy}\label{ss:tables}
Proposition \ref{p:unip ind} allows us to determine the unipotent classes $\unu_{\bP}$ in $\dualG^{\sigma}$ in all cases. We shall assume $\chk$ is large so that $\bP$ corresponds to a unique regular elliptic number $m=m(\bP)$ of $(\WW,\sigma)$. We denote $\unu_{\bP}$ by $\unu_{m}$. For classical groups, the classification of admissible parahorics will be reviewed in \S\ref{s:u}-\S\ref{s:o}. The notation $\boxed{i}\times j$ denotes $j$ Jordan blocks each of size $i$.

\begin{tabular}{|l|l|l|l|}
\hline
$G$ & $m$ & $\unu_{m}$ & $\dualG^{\sigma}$  \\ \hline
$A_{n-1}$ & $n$ & $\boxed{n}$ & $A_{n-1}$ \\ \hline
$\leftexp{2}A_{n-1}, n$ odd & $m=\frac{2n}{d}$; $d|n$ & $\boxed{\frac{m}{2}}\times d$ & $B_{\frac{n-1}{2}}$ \\ 
 & $m=\frac{2(n-1)}{d}$; $\frac{n-1}{d}$ odd, $d<n-1$ & $\boxed{1}+\boxed{\frac{m}{2}}\times d$ &  \\ \hline
$\leftexp{2}A_{n-1}, n$ even & $m=\frac{2n}{d}$; $\frac{n}{d}$ odd & $\boxed{\frac{m}{2}-1}+\boxed{\frac{m}{2}}\times (d-2)+\boxed{\frac{m}{2}+1}$ & $C_{\frac{n}{2}}$ \\ 
 & $m=\frac{2(n-1)}{d}$; $d|n-1, d<n-1$ & $\boxed{\frac{m}{2}}\times (d-1)+\boxed{\frac{m}{2}+1}$ &  \\ \hline 
$B_{n}$ & $m=2n/d; d|n$ & $\boxed{m}\times d$ & $C_{n}$ \\ \hline
$C_{n}$ & $m=2n/d; d|n, d$ odd & $\boxed{m}\times(d-1)+\boxed{m+1}$ & $B_{n}$ \\  
 & $m=2n/d; d|n, d$ even & $\boxed{1}+\boxed{m-1}+\boxed{m}\times(d-2)+\boxed{m+1}$ & \\ \hline 
$D_{n}$ & $m=2n/d; d|n, d$ even & $\boxed{m-1}+\boxed{m}\times(d-2)+\boxed{m+1}$ & $D_{n}$ \\ 
 & $m=\frac{2(n-1)}{d}; d|n-1, d$ odd & $\boxed{1}+\boxed{m}\times(d-1)+\boxed{m+1}$ &  \\ \hline
$\leftexp{2}{D}_{n}$ & $m=2n/d; d|n, d$ odd & $\boxed{m-1}+\boxed{m}\times(d-1)$ & $B_{n-1}$ \\
& $m=\frac{2(n-1)}{d}; d|n-1, d$ even & $\boxed{1}+\boxed{m}\times d$ & \\
\hline
\end{tabular}

For exceptional groups, we use the tables in \cite[\S7.1]{GLRY} to list all regular elliptic numbers, and use Bala-Carter's notation \cite[\S13.1]{Carter} to denote unipotent classes. In many cases, the fact that $\dim\calB_{u}=\ell(w_{\bP})$ is already enough to determine the class $\unu_{\bP}$. In the remaining cases $G$ is always split. 
\begin{itemize}
\item When $\bP\subset\GG(\calO_{K})$, we use Proposition \ref{p:unip ind} to conclude that $\unu_{\bP}$ is the $S$-distinguished unipotent classes assigned to these admissible parahoric subgroups in \cite[\S7.3]{GLRY}. The weighted Dynkin diagram of the distinguished unipotent class $\unu_{\bP}$ is obtained by putting $0$ on simple roots in $L_{\bP}$ and putting $2$ elsewhere.
\item For those $\bP$ not contained in $\GG(\calO_{K})$, we first use the tables in \cite[\S7.1]{Lsp} to find the dimension of the truncated induction $j^{W}_{W_{\bP}}(\ep)$, and then use the tables in \cite[\S13.3]{Carter} to find the corresponding unipotent classes.
\end{itemize}

\begin{tabular}{|l|l|l|l||l|l|l|l||l|l|l|l|}
\hline
$G$ & $m$ & $\unu_{m}$ & $\dualG^{\sigma}$ & $G$ & $m$ & $\unu_{m}$ & $\dualG^{\sigma}$ & $G$ & $m$ & $\unu_{m}$ & $\dualG^{\sigma}$ \\ \hline

$E_{6}$ & $3$ & $2A_{2}+A_{1}$ & $E_{6}$ & 
   $\leftexp{3}{D}_{4}$ & 3 & $A_{1}$ & $G_{2}$ &
        $E_{8}$ & 2 & $4A_{1}$ & $E_{8}$ \\
& $6$ & $E_{6}(a_{3})$ & &
   & 6 & $G_{2}(a_{1})$ & &
	& 3 & $2A_{2}+2A_{1}$ &\\
& $9$ & $E_{6}(a_{1})$ & &
   & 12 & $G_{2}$ &  & 
   	& 4 & $2A_{3}$ &\\ \cline{5-8}
& $12$ & $E_{6}$ & &  
   $F_{4}$ & 2 & $A_{1}+\wt{A}_{1}$ & $F_{4}$ &
	& 5 & $A_{4}+A_{3}$ &\\ \cline{1-4}
$\leftexp{2}{E}_{6}$ & 2 & $A_{1}$ & $F_{4}$ &
   & 3 &  $\wt{A}_{2}+A_{1}$ & &
 	& 6 & $E_{8}(a_{7})$ &\\
& 4 & $A_{2}+\wt{A}_{1}$ &  &
   & 4 & $F_{4}(a_{3})$ &&
 	& 8 & $A_{7}$ &\\
& 6 & $F_{4}(a_{3})$ &  &
   & 6 & $F_{4}(a_{2})$ &&
 	& 10 & $E_{8}(a_{6})$ &\\
& 12 & $F_{4}(a_{1})$ &  &
   & 8 & $F_{4}(a_{1})$ &&
	& 12 & $E_{8}(a_{5})$ &\\
& 18 & $F_{4}$ & &  
   & 12 & $F_{4}$ && 
	& 15 & $E_{8}(a_{4})$ & \\ \cline{1-8}
$E_{7}$ & 2 & $4A_{1}$ & $E_{7}$ &
   $G_{2}$ & 2 & $\wt{A}_{1}$ & $G_{2}$ & 
	& 20 & $E_{8}(a_{2})$ &\\
& 6 & $E_{7}(a_{5})$ & & 
   & 3& $G_{2}(a_{1})$ & &
	& 24 & $E_{8}(a_{1})$ &\\
& 14 & $E_{7}(a_{1})$ & & 
   & 6 & $G_{2}$ & &
	& 30 & $E_{8}$ &\\ \cline{5-12}
& 18 & $E_{7}$ & \\ \cline{1-4}
\end{tabular}

\subsection{Conjectural description of the tame monodromy for general $\chi$}\label{ss:localmono}
Finally we give a conjectural description of the local monodromy of $\Kl_{\dG,\bP}(\chi,\phi)$ at $0$ for general $\chi:\tLab_{\bP}\to \Ql^{\times}$ and quasi-split $G$. First of all $\Kl_{\dG,\bP}(\chi, \phi)$ is tame at $0$. Recall the tame inertia $\calI^{t}_{0}\cong\varprojlim_{(n,p)=1}\mu_{n}(\bark)$ and in particular $\calI^{t}_{0}\surj \mu_{e}(k)$. Let $\xi\in\calI^{t}_{0}$ be a generator that  maps to $\sigma\in\mu_{e}$. We predict that under the monodromy representation $\calI^{t}_{0}\to\dG$ associated with $\Kl_{\dG,\bP}(\chi,\phi)$, $\xi$ should go to $(\kappa u, \sigma)\in\dualG^{\sigma}\times{\sigma}\subset\dG$ where $\kappa u$ is the Jordan decomposition into the semisimple part and unipotent part of an element in $\dualG^{\sigma}$. 

We first describe the semisimple part $\kappa$. Up to $\dualG$-conjugacy we may arrange that $\kappa\in\dualT$. By further $\dualT$-conjugation, only the image of $\kappa$ in the $\sigma$-coinvariants $\dualT_{\sigma}$ is well-defined. The character $k^{\times}\otimes\xcoch(\TT)^{\sigma}=\SS(k)\to \Lab_{\bP}\xrightarrow{\chi}\Ql^{\times}$ gives a homomorphism $\chi': k^{\times}\to\dualT_{\sigma}=\Hom(\xcoch(\TT)^{\sigma}, \Ql^{\times})$. Local class field theory identifies $k^{\times}$ with the tame inertia of $\Gal(F^{\textup{ab}}_{0}/F_{0})$ (where $F^{\textup{ab}}_{0}$ is a maximal abelian extension of $F_{0}$), hence there is a canonical surjection $\calI^{t}_{0}\surj k^{\times}$. Composing with $\chi'$ we get a homomorphism $\chi'':\calI^{t}_{0}\to\dualT_{\sigma}$. The image of $\kappa$ in $\dualT_{\sigma}$ should be $\chi''(\xi)$. This should be provable using similar techniques as in the proof of \cite[Proposition 5.4]{Ymotive}, based on results in \cite{Be}.

Next we describe the unipotent part $u$. Let $W_{\kappa,\aff}\subset\Wa$ be the subgroup generated by affine reflections whose vector part $\alpha\in\Phi$ is such that the composition $\Gm(k)\xrightarrow{\alpha^{\vee}}\SS(k)\to \Lab_{\bP}(k)\xrightarrow{\chi}\Ql^{\times}$ is trivial. It can be shown that $W_{\kappa,\aff}$ is an affine Coxeter group. In fact $W_{\kappa,\aff}$ can be identified with the affine Weyl group of the group dual to $\dualG^{\sigma,\circ}_{\kappa}$, the centralizer in $\dualG^{\sigma,\circ}$ of an element $\wt\kappa\in\dualT^{\sigma,\circ}$ lifting $\kappa\in\dualT_{\sigma}$. The longest element $w_{\bP}\in W_{\bP}$ belongs to $W_{\kappa,\aff}$, hence lies in a two-sided cell $\unc_{\kappa, \bP}$ of $W_{\kappa,\aff}$.  By Lusztig's Theorem \ref{th:Lu}, $\unc_{\kappa, \bP}$ corresponds to a unipotent class $\unu_{\kappa, \bP}$ of $\dualG^{\sigma,\circ}_{\kappa}$. The unipotent element $u$ should lie in $\unu_{\kappa,\bP}$. This should be provable, in the same way as Theorem \ref{th:unip}, using the results of Bezrukavnikov {\em et al.} in \cite{Be2}, which are monodromic analogue of his results in \cite{Be}.

\subsection{Hecke operators at $0$}\label{ss:Hk0} 
The rest of this section is devoted to the proof of Theorem \ref{th:unip}. We first set up some notation. Fix a stable functional $\phi:V_{\bP}\to k$. Since $\chi=1$, we simply denote the Kloosterman sheaf $\Kl_{\dualG,\bP}(1,\phi)$ by $\Kl_{\dualG, \bP}(\phi)$. Comparing a general split $G$ with its simply-connected cover $G^{\sc}$, it is easy to see that the local system $\Kl_{\dualG, \bP}(\phi)$ induces the local system $\Kl_{\widehat{G^{\sc}}, m}(\phi)$ via the map $\dualG\to \widehat{G^{\sc}}=\dualG^{\ad}$. Therefore it suffices to prove Theorem \ref{th:unip} for simply-connected $G$. 

In the sequel we shall base change all spaces to $\bark$ without changing notation. When $S$ is a point $\phi: \Spec\ \bark\to V^{*,\st}_{\bP}$, we denote the category of automorphic sheaves $\calD(1, S)$ by $\calD(\phi)$. Let $A(\phi)\in \calD(\phi)$ be the automorphic sheaf $A(1, S)$ in Lemma \ref{l:clean}(1), which is a Hecke eigensheaf with eigen local system $\Kl_{\dualG,\bP}(\phi)$ by Theorem \ref{th:main} .

Geometric Hecke operators on $\calD(\phi)$ that arise from modifying $G$-bundles at $0$ can be introduced using a diagram similar to \eqref{Hk}. For details we refer to \cite[\S5.1.1]{Ymotive}. We have a functor
\begin{equation*}
T_{0}(-,-):D^{b}_{c}(\bP_{0}\backslash L_{0}G/\bP_{0})\times \calD(\phi)\to \calD(\phi)
\end{equation*}
There is a monoidal structure on $D^{b}_{c}(\bP_{0}\backslash L_{0}G/\bP_{0})$ given by convolution of sheaves $(C_{1},C_{2})\mapsto C_{1}\conv{\bP_{0}}C_{2}$ with unit object $\delta_{\bP_{0}}$ given by the constant sheaf supported on the identity double coset. This monoidal structure is compatible with the composition of geometric Hecke operators: there is natural isomorphism of endo-functors on $\calD(\phi)$
\begin{equation*}
T_{0}(C_{1}, T_{0}(C_{2}, -))\cong T_{0}(C_{1}\conv{\bP_{0}}C_{2}, -)
\end{equation*}
which satisfies associativity. By Lemma \ref{l:clean}(1),  $\calD(\phi)\cong D^{b}_{c}(\Vect)$ with the inverse functor given by $(-)\otimes A(\phi)$,  therefore we have $T_{0}(C,A(\phi))\cong e(C)\otimes A(\phi)$ for some monoidal functor
\begin{equation*}
e: D^{b}_{c}(\bP_{0}\backslash L_{0}G/\bP_{0})\to D^{b}_{c}(\Vect).
\end{equation*}

One key idea leading to the proof of Theorem \ref{th:unip} is the following simple observation.
\begin{lemma}\label{l:exact} The functor $e$ is exact with respect to the perverse t-structure on the source and the usual t-structure on the target.
\end{lemma}
\begin{proof} Let $\omega: \Fl_{\bP_{0}}=L_{0}G/\bP_{0}\to \Bun_{G}(\bP_{0}, \bP^{++}_{\infty})$ be the uniformization map. Recall from Lemma \ref{l:clean} that any object in $\calD(\phi)$ is supported on the open subset $U\subset\Bun$. We think of $\Fl_{\bP_{0}}$ as the fiber preimage of $\star\in\Bun$ under $\orr{h}_{0}:\Hk_{0}\to \Bun$, and think of $\omega$ as the restriction of $\oll{h}_{0}:\Hk_{0}\to\Bun$, where $\Hk_{0}$ is the analog of $\Hk$ that only modifies the $G$-bundles at $0$. By construction, $e(C)$ is the stalk at $\star$ of $T(C,A(\phi))=\orr{h}_{0,!}(\oll{h}^{*}_{0}A(\phi)\otimes C^{\Hk})$. By proper base change we have
\begin{equation}\label{ee}
e(C)\cong \cohoc{*}{\omega^{-1}(U), \omega^{*}A(\phi)\otimes C}\cong \cohog{*}{\omega^{-1}(U), \omega^{*}A(\phi)\otimes C}
\end{equation}
where the second isomorphism is proved by replacing $C$ by its Verdier dual $\DD(C)$ in the first isomorphism and using Lemma \ref{l:clean}(1). 

If $C\in D^{b}_{c}(\bP_{0}\backslash L_{0}G/\bP_{0})$ is perverse, then $\omega^{*}A(\phi)\otimes C$ is a perverse sheaf on $\omega^{-1}(U)$. By \eqref{ee} and the t-exactness properties of direct image functors under an affine morphism,  it suffices to show that $\omega^{-1}(U)\subset\Fl_{\bP_{0}}$ is affine. However, using the Birkhoff decomposition \eqref{Birk} (with the role of $0$ and $\infty$ swapped), we have $\omega^{-1}(U)\cong \Gamma_{\infty}^{+}=(\Res^{k[t,t^{-1}]}_{k}G)\cap \bP^{+}_{\infty}$, which is an affine algebraic group.
\end{proof}

\subsection{Gaitsgory's nearby cycles functor} Let $\calP:=\Perv(\bP_{0}\backslash L_{0}G/\bP_{0})\subset D^{b}_{c}(\bP_{0}\backslash L_{0}G/\bP_{0})$ be the category of perverse sheaves (the perverse t-structures is inherited from the category of $\Ql$-complexes on the affine partial flag variety $L_{0}G/\bP_{0}$).

There is a parahoric analogue Gaitsgory's nearby cycles functor (see \cite[\S5.1.2]{Ymotive})
\begin{equation*}
Z_{\bP_{0}}: \Rep(\dualG)\to \calP.
\end{equation*}
This is an exact central functor (see \cite[\S2, Definition 1]{Be}) that admits an action of the inertia group $\calI_{0}$ because it comes from nearby cycles. This action factors through the tame quotient $\calI^{t}_{0}$ and is unipotent. By \cite[Lemma 5.1]{Ymotive}, there is a canonical $\calI_{0}$-equivariant isomorphism
\begin{equation*}
T_{0}(Z_{\bP_{0}}(V),A(\phi))\cong (\Kl^{V}_{\dualG,\bP}(\phi))|_{\Spec F^{s}_{0}}\otimes A(\phi)
\end{equation*}
that is compatible with the monoidal structures. Here $\calI_{0}$ acts on $Z_{\bP_{0}}(V)$ by the above remark and on the geometric stalk $(\Kl^{V}_{\dualG,\bP}(\phi))|_{\Spec F^{s}_{0}}$ ($F^{s}_{0}$ is a separable closure of the local field $F_{0}$). In other words, we have an isomorphism of tensor functors
\begin{equation}\label{Zmono}
e\circ Z_{\bP_{0}}(-)\cong \Kl_{\dualG, \bP}(\phi)|_{\Spec F^{s}_{0}}: \Rep(\dualG)\to \Vect
\end{equation}
that is equivariant under the actions of $\calI_{0}$.

\subsection{Serre quotients of $\calP$}
The category $\calP=\Perv(\bP_{0}\backslash L_{0}G/\bP_{0})$ is not closed under the convolution $\conv{\bP_{0}}$ in general. There is a universal way to fix this problem. Let $\calN\subset\calP$ be the Serre subcategory of $\calP$ generated (under extensions) by irreducible objects that appear as simple constituents of $\pH^{i}(C_{1}\conv{\bP_{0}}C_{2})$ for some $C_{1}, C_{2}\in\calP$ and $i\neq0$. Consider the Serre quotient $\bcP:=\calP/\calN$.
 
\begin{lemma}
The functor $\pH^{0}(-\conv{\bP_{0}}-):\calP\times\calP\to\calP\surj\bcP$ factors through $\bcP\times\bcP$ and and it defines a monoidal structure $\circledast: \bcP\times\bcP\to \bcP$ on $\bcP$. \end{lemma}
\begin{proof}
We show that if an irreducible object $C\in\calN$, then $\pH^{0}(C\conv{\bP_{0}}C')\in\calN$ for any $C'\in\calP$. By definition $C$ is a subquotient of $\pH^{i}(C_{1}\conv{\bP_{0}}C_{2})$ for some $i\neq0$ and $C_{1}, C_{2}$ simple objects in $\calP$. By the decomposition theorem, $C$ is in fact a direct summand of $\pH^{i}(C_{1}\conv{\bP_{0}}C_{2})$. Therefore it suffices to show that $\pH^{0}(\pH^{i}(C_{1}\conv{\bP_{0}}C_{2})\conv{\bP_{0}}C')\in\calN$. Again by the decomposition theorem, the spectral sequence $\pH^{j}(\pH^{i}(C_{1}\conv{\bP_{0}}C_{2})\conv{\bP_{0}}C')\Rightarrow \pH^{i+j}(C_{1}\conv{\bP_{0}}C_{2}\conv{\bP_{0}}C')$ degenerates at $E_{2}$, therefore $\pH^{0}(\pH^{i}(C_{1}\conv{\bP_{0}}C_{2})\conv{\bP_{0}}C')$ is a subquotient of $\pH^{i}(C_{1}\conv{\bP_{0}}C_{2}\conv{\bP_{0}}C')$, hence in $\calN$ since $i\neq0$. This shows that the functor $\pH^{0}(-\conv{\bP_{0}}-)$ factors through $\bcP\times\bcP$. The associativity of $\circledast$ can be proved in the same way. 
\end{proof}

Let $a$ be the length of the longest element of $W_{\bP}$, so we have $a=a(\unc)$ by Lemma \ref{l:aw}. Let $\pi: \Fl=L_{0}G/\bI_{0}\to \Fl_{\bP_{0}}=L_{0}G/\bP_{0}$ be the projection. Let $\calQ=\Perv(\bI_{0}\backslash L_{0}G/\bI_{0})$. Since $\pi$ is smooth of relative dimension $a$, the functor $\pi^{\natural}:=\pi^{*}[a](a/2)$ gives an exact functor $\calP\to\calQ$ that is also fully faithful (we fix a square root of $q=\#k$, hence the half Tate twist makes sense). We may therefore identify $\calP$ with a full subcategory of $\calQ$ via the functor $\pi^{\natural}$.

The category $\calQ$ carries a filtration by full subcategories indexed by two-sided cells in $\Wa$. Let $\unc=\unc_{\bP}$ be the cell containing the longest element in $W_{\bP}$. Then $\calP\subset \calQ_{\leq \unc}$. The category $\calP_{<\unc}:=\calP\cap\calQ_{<\unc}$ is a Serre subcategory (generated by those irreducible objects in $\calP$ indexed by longest representatives of $(W_{\bP},W_{\bP})$-double cosets in $\Wa$ that belong to a cell smaller than $\unc$). Let $\calP_{\unc}:=\calP/\calP_{<\unc}$ be the Serre quotient. This is a full subcategory of a similar Serre quotient $\calQ_{\unc}:=\calQ_{\leq\unc}/\calQ_{<\unc}$. It is known that $\calQ_{\unc}$ carries a monoidal structure $\odot$ given by truncated convolution: $\tilC_{1}\odot \tilC_{2}$ is the image of $\pH^{a}(\tilC_{1}\conv{\bI_{0}}\tilC_{2})(a/2)$ in $\calQ_{\unc}$. 

\begin{lemma} The category $\calP_{\unc}\subset\calQ_{\unc}$ is closed under the truncated convolution $\odot$. The quotient functor $\calP\surj\calP_{\unc}$ factors through $\bcP$ and induces a monoidal Serre quotient functor
\begin{equation*}
\eta: (\bcP, \circledast)\surj(\calP_{\unc}, \odot)
\end{equation*}
\end{lemma}
\begin{proof}
For $C_{1},C_{2}\in\calP$ with image $\tilC_{1}$ and $\tilC_{2}$ in $\calQ$, we have
\begin{equation}\label{two conv}
\pi^{\natural}C_{1}\conv{\bI_{0}}\pi^{\natural}C_{2}\cong\pi^{\natural}(C_{1}\conv{\bP_{0}}C_{2})\otimes\cohog{*}{\bP_{0}/\bI_{0}}[a](a/2).
\end{equation}

We first show that $\calN\subset \calP_{<\unc}$ so the quotient functor factors as $\bcP\surj\calP_{\unc}$. For a simple perverse sheaf $C_{w}\in\calN$ indexed by $w\in W_{\bP}\backslash \Wa/W_{\bP}$, it appears in $\pH^{i}(C_{1}\conv{\bP_{0}}C_{2})$ for some simple perverse sheaves $C_{1}, C_{2}\in\calP$ and some $i\neq0$. Since $C_{1}\conv{\bP_{0}}C_{2}$ is isomorphic to its own Verdier dual, we may assume $i>0$. Using \eqref{two conv} and taking the top cohomology of $\cohog{*}{\bP_{0}/\bI_{0}}$, we conclude that $\pi^{\natural}\pH^{i}(C_{1}\conv{\bP_{0}}C_{2})$ is a direct summand of $\pH^{a+i}(\pi^{\natural}C_{1}\conv{\bI_{0}}\pi^{\natural}C_{2})(a/2)$. Therefore $\pi^{\natural}C_{w}$ also appears in $\pH^{a+i}(\pi^{\natural}C_{1}\conv{\bI_{0}}\pi^{\natural}C_{2})(a/2)$. Now $\pi^{\natural}C_{w}$ is $\IC_{\tilw}$ (up to a Tate twist) for $\tilw$ the longest representative of the coset $w$. By the definition of the $a$-function, this means $a(\tilw)\geq a+i>a$, hence $\tilw$ belongs to a smaller cell than $\unc$ (we have a priori $\tilw\leq\unc$). This shows that $\pi^{\natural}C_{w}\in\calQ_{<\unc}$, hence $C_{w}\in\calP_{<\unc}$. The conclusion is that $\calN\subset\calP_{<\unc}$ and we have a Serre quotient functor $\bcP\surj \calP_{\unc}$.

Taking $a$-th perverse cohomologies of both sides of \eqref{two conv}, we get
\begin{equation*}
\pH^{a}(\pi^{\natural}C_{1}\conv{\bI_{0}}\pi^{\natural}C_{2})\cong\bigoplus_{i\geq0}\pi^{\natural}\pH^{i}(C_{1}\conv{\bP_{0}}C_{2})\otimes\cohog{2a-i}{\bP_{0}/\bI_{0}}(a/2).
\end{equation*}
For $i>0$ the corresponding term on the right side lies in $\calN$, hence have zero image in $\calP_{\unc}$ by the previous discussion. Therefore $\pi^{\natural}C_{1}\odot\pi^{\natural}C_{2}$, defined as the image of $\pH^{a}(\pi^{\natural}C_{1}\conv{\bI_{0}}\pi^{\natural}C_{2})(a/2)$ in $\calQ_{\unc}$, is equal to the image of $\pi^{\natural}\pH^{0}(C_{1}\conv{\bP_{0}}C_{2})\otimes\cohog{2a}{\bP_{0}/\bI_{0}}(a)\cong\pi^{\natural}\pH^{0}(C_{1}\conv{\bP_{0}}C_{2})$ in $\calP_{\unc}$, which is the same as the image of $C_{1}\circledast C_{2}\in\bcP$ in $\calP_{\unc}$. This shows that $\calP_{\unc}$ is closed under $\odot$, and that the quotient functor $\bcP\to\calP_{\unc}$ is monoidal.
\end{proof}

\subsection{Finish of the proof of Theorem \ref{th:unip}}\label{ss:pf unip} By Lemma \ref{l:exact}, $e$ is a t-exact functor, and therefore $e|_{\calP}$ factors through the Serre quotient $\bcP$ by the universal property of $\bcP$. We get a monoidal functor $\overline{e}:(\bcP,\circledast)\to(\Vect,\otimes)$. 

Denote by $\barZ$ the composition $\Rep(\dualG)\xrightarrow{Z_{\bP_{0}}}\calP\surj\bcP$, which is an exact monoidal functor. Let $\bcP'\subset\bcP$ be the full subcategory consisting of subquotients of the images of $\barZ$. Similarly, let $\calP'_{\unc}\subset \calP_{\unc}$ be the full subcategory consisting of subquotient of images of $\eta\circ\barZ$. 

Summarizing, we have the following commutative diagram
\begin{equation*}
\xymatrix{\Rep(\dualG)\ar[r]^{\barZ_{\bP_{0}}} & \bcP'\ar@{^{(}->}[r]\ar@{->>}[d]^{\eta'} & \bcP\ar@{->>}[d]^{\eta}\ar[r]^{\overline{e}} & \Vect\\
 & \calP'_{\unc}\ar@{^{(}->}[r] & \calP_{\unc}}
\end{equation*}
Both functors $\barZ$ and $\eta'\circ\barZ$ are exact central functors. Applying a criterion of Bezrukavnikov in \cite[Proposition 1]{Be} to both functors, we see that both $\bcP'$ and $\calP'_{\unc}$ are neutral Tannakian categories over $\Ql$ and hence the tensor functor  $\eta'$ between them  must be faithful (\cite[Proposition 1.4]{DM}). However $\eta'$ is also a Serre quotient functor. This forces $\eta'$ to be an equivalence of tensor categories.  Write $\bcP'\cong\Rep(\dualH)$ for some algebraic subgroup $\dualH\subset\dualG$ so that $\barZ$ becomes the restriction functor from $\dualG$ to $\dualH$. Therefore we may identify the two rows in the above diagram and obtain a factorization of the fiber functor $e\circ Z_{\bP_{0}}$ as
\begin{equation*}
e\circ Z_{\bP_{0}}:\Rep(\dualG)\xrightarrow{\Res^{\dualG}_{\dualH}} \Rep(\dualH)\cong \calP'_{\unc}\cong\bcP'\xrightarrow{\overline{e}}\Vect.
\end{equation*}
Note that the category $\calP'_{\unc}$ (viewed as a subcategory of $\calQ_{\unc}$) is the category $\calA_{w_{\bP}}$ introduced by Bezrukavnikov in \cite[\S4.3]{Be}, where $w_{\bP}\in W_{\bP}$ is the longest element, and is a Duflo (or distinguished) involution (\cite[II, Remark following Proposition 1.4]{cells}) by Lemma \ref{l:aw}. Fix a topological generator $\xi\in \calI^{t}_{0}$, then $\xi$ acts on the restriction functor $\Res^{\dualG}_{\dualH}$. By Tannakian formalism, the action of $\xi$ on the fiber functor $\overline{e}\circ\Res^{\dualG}_{\dualH}=e\circ Z_{\bP_{0}}$ gives an element $u\in \dualG$ that commutes with $\dualH$ (see \cite[Theorem 1]{Be}). By \eqref{Zmono}, $u$ is conjugate to the image of $\xi$ under the monodromy representation of the local system $\Kl_{\dualG,\bP}(\phi)$. Therefore we need to prove that $u\in\unu_{\bP}$, i.e., $u$ corresponds to $\unc_{\bP}$ under Lusztig's bijection in Theorem \ref{th:Lu}. But this is exactly what was shown by Bezrukavnikov in \cite[Theorem 2]{Be}. This completes the proof of Theorem \ref{th:unip}.

\section{Rigidity}\label{s:rigidity}
In this section we shall deduce that $\Kl_{\dualG,\bP}(1,\phi)$ is cohomologically rigid under some additional assumptions. We work over the base field $\bark$ without changing notation. 

\subsection{Cohomological rigidity} Recall that a $\dualG$-local system $E_{\dualG}$ over an open subset $\pline$ of a complete smooth connected algebraic curve $X$ is called {\em cohomologically rigid} if
\begin{equation*}
\cohog{*}{X,j_{!*}E^{\Ad}_{\dualG}}=0
\end{equation*}
where $j:\pline\incl X$ is the inclusion, $E^{\Ad}_{\dualG}$ is the local system on $\pline$ associated with the adjoint representation of $\dualG$ on its Lie algebra $\dualg$, and $j_{!*}$ simply means the sheaf-theoretic direct image $\upH^{0}j_{*}$.

Back to the situation of \S\ref{ss:auto} with $\chi=1$. We write $\Kl_{\dualG,\bP}(1, \phi)$ as $\Kl_{\dualG,\bP}(\phi)$. Let $j:\pline\incl X$ be the open inclusion, and $i_{0}, i_{\infty}:\Spec\ \bark\incl X$ be the closed embedding of the two points $0$ and $\infty$.

\begin{prop} Assume $\chi=1$ and $G$ is split. Assume $\chk$ is sufficiently large so that $m=m(\bP)$ is a regular elliptic number for $\WW$ (see \S\ref{ss:bij}). Assuming the properties (1) and (2) in \S\ref{ss:wildL} about the monodromy of $\Kl_{\dualG,\bP}(\phi)$ at $\infty$ hold. Then we have
\begin{equation*}
\cohog{*}{X,j_{!*}\Kl^{\Ad}_{\dualG,\bP}(\phi)}=0.
\end{equation*}
I.e., $\Kl_{\dualG, \bP}(\phi)$ is cohomologically rigid.
\end{prop}
\begin{proof} Since $\dualg^{\calI_{\infty}}=0$ by property (2) of \S\ref{ss:wildL}, $\Kl^{\Ad}_{\dualG,\bP}(\phi)$ admits no global sections, i.e., $\cohog{0}{\pline, \Kl^{\Ad}_{\dualG,\bP}(\phi)}=\cohog{0}{X, j_{!*}\Kl^{\Ad}_{\dualG,\bP}(\phi)}=0$. Dually $\cohoc{2}{\pline, \Kl^{\Ad}_{\dualG,\bP}(\phi)}=\cohog{2}{X, j_{!*}\Kl^{\Ad}_{\dualG,\bP}(\phi)}=0$.

The distinguished triangle $j_{!}\to j_{!*}\to i_{0,*}\upH^{0}i^{*}_{0}\oplus i_{\infty,*}\upH^{0}i_{\infty}^{*}\to j_{!}[1]$ gives an exact sequence
\begin{equation}\label{hh1}
0\to \dualg^{\calI_0}\oplus\dualg^{\calI_\infty}\to \cohoc{1}{\pline,\Kl^{\Ad}_{\dualG,\bP}(\phi)}\to \cohog{1}{X,j_{!*}\Kl^{\Ad}_{\dualG,\bP}(\phi)}\to0.
\end{equation}
The Grothendieck-Ogg-Shafarevich formula implies
\begin{equation*}
\chi_c(\pline,\Kl^{\Ad}_{\dualG, \bP}(\phi))=-\Swan_\infty(\dualg).
\end{equation*}
Using the vanishing of $\upH^{2}_{c}$ and property (2) of \S\ref{ss:wildL}, we get
\begin{equation*}
\dim \cohoc{1}{\pline,\Kl^{\Ad}_{\dualG,\bP}(\phi)}=\#\Phi/m.
\end{equation*}
On the other hand, by Theorem \ref{th:unip}, $\dualg^{\calI_0}$ is the dimension of the centralizer of a unipotent element $u$ in the unipotent conjugacy class $\unu_{\bP}$ in $\dualG$. By Steinberg's theorem, $\dim\dualg^{\calI_0}=\dim\dualg^{u}=2\dim\calB_{u}+r$ (where $r$ is the rank of $G$). By Theorem \ref{th:Lu}(1), $\dim\calB_{u}=a(\unc_{\bP})=\ell(w_{\bP})$. Hence $\dim\calB_{u}=\#\Psi(L_{\bP})/2$ ($\Psi(L_{\bP})$ is the root system of $L_{\bP}$), and
\begin{equation*}
\dim\dualg^{\calI_0}=2\dim\calB_{u}+r=\#\Psi(L_{\bP})+r=\dim L_{\bP}.
\end{equation*}
By the dictionary between stable gradings on Lie algebras and regular elements in Weyl groups given in \cite[\S4]{GLRY}, $L_{\bP}$ is isomorphic to the neutral component of the fixed point subgroup of an automorphism of $\GG$ given by a lifting of a regular elliptic element $w\in \WW$ of order $m$. By analyzing the action of $w$ on the Lie algebra of $\GG$ it is easy to see that $\dim L_{\bP}=\#\Phi/m$. Therefore the first two terms in the exact sequence \eqref{hh1} both have dimension $\#\Phi/m$, hence the third term vanishes. This proves the vanishing of $\cohoc{*}{X,j_{!*}\Kl^{\Ad}_{\dualG,\bP}(\phi)}$ in all degrees.
\end{proof}

\begin{remark} When $\bP=\bI$, we computed the Euler characteristic of $\Kl^{\Ad}_{\dualG,\bI}(\phi)$ in \cite[Theorem 4]{HNY} and used this to confirm the predictions about the wild monodromy at $\infty$ in \S\ref{ss:wildL} in the Iwahori case (in this case, these predictions were made in \cite{GR}). However, the computation in \cite{HNY} is very complicated (see the argument of Proposition \ref{p:ucal} to get a sense). In principle, similar but even more complicated computation could be done for general admissible parahoric subgroups $\bP$, and the result would confirm the predictions in \S\ref{ss:wildL} (at least when $\unu_{\bP}$ is a distinguished unipotent class, in which case $\Kl^{\Ad}_{\dualG,\bP}(\phi)$ does not have global sections).
\end{remark}

\section{Examples: Quasi-split unitary groups}\label{s:u}
When $G$ is split of type $A$, the only admissible standard parahoric subgroup is the Iwahori $\bI$. In this case, $\Kl_{\dG, \bI}(\chi,\phi)$ has been calculated in \cite[\S3]{HNY}, and is identified with the Kloosterman sheaves defined by Deligne. In section, we shall compute the generalized Kloosterman sheaves for the quasi-split unitary group $G$.

\subsection{Linear algebra}\label{ulinear} Assume $\chk\neq2$. Let $(M,q)$ be a quadratic space over $k$ of dimension $n$. Let $(\cdot,\cdot):M\times M\to k$ be the associated symmetric bilinear form $(x,y)=q(x+y)-q(x)-q(y)$. Let $K'$ be the ramified quadratic extension of $K$ with uniformizer $\varpi^{1/2}$. Then there is a Hermitian form $h$ on $M\otimes_{k}K'$ defined as
\begin{equation}\label{herm}
h(x+\varpi^{1/2}y, z+\varpi^{1/2}w)=(x,y)-\varpi(y,w)+\varpi^{1/2}(y,z)-\varpi^{1/2}(x,w)
\end{equation}
for $x,y,z,w\in M\otimes_{k}K$.

Let $G=\Ug(M\otimes K',h)$ be the unitary group of the Hermitian space $(M\otimes K', h)$. This corresponds to the absolute group $\GG=\GL(M)$ and a nontrivial outer automorphism $\sigma$. The regular elliptic numbers $m$ of $(\WW,\sigma)$ are in bijection with divisors $d|n$ or $d|n-1$ (in which case we require $d<n-1$) such that the quotient $n/d$ or $(n-1)/d$ is odd. We have $m=2n/d$ or $m=2(n-1)/d$ in the two cases. Since $m/2$ is always odd, we write $m/2=2\ell+1$. For $d$ and $m$ as above, we fix a decomposition
\begin{equation}\label{ud}
M=M_{-\ell}\oplus M_{-1}\oplus M_{0}\oplus M_{1}\oplus\cdots\oplus M_{\ell}
\end{equation}
such that $d=\dim M_{i}\leq \dim M_{0}\leq d+1$ (for all $i=\pm1,\cdots, \pm\ell$), and $(M_{i}, M_{j})=0$ unless $i+j=0$. Denote the restriction of $q$ to $M_{0}$ by $q_{0}$. The pairing $(\cdot,\cdot)$ identifies $M_{-j}$ with $M_{j}^{*}$.

Let $\wt\bP_{m}\subset G(K)$ be the stabilizer of the lattice chain $\L_{\ell}\supset \L_{\ell-1}\supset\cdots\supset\L_{-\ell}$ defined as
\begin{equation*}
\L_{i}=\bigoplus_{-\ell\leq j\leq i}M_{j}\otimes \calO_{K'}\oplus\bigoplus_{i<j\leq \ell} M_{j}\otimes\varpi^{1/2}\calO_{K'}.
\end{equation*}
Its reductive quotient is $\tL_{m}=\prod_{i=1}^{\ell}\GL(M_{i})\times\Og(M_{0},q_{0})$, where $\GL(M_{i})$ acts on $M_{i}$ by the standard representation and on $M_{-i}=M_{i}^{*}$ by the dual of the standard representation. We have
\begin{equation}\label{uLab}
\tLab_{m}\cong \prod_{i=1}^{\ell}\Gm\times\{\pm1\}
\end{equation}
by taking determinants. The subgroup $\bP_{m}\subset\wt\bP_{m}$, defined as the kernel of $\wt\bP_{m}\to \tL_{m}\to\tLab_{m}\surj\{\pm1\}$, is an admissible parahoric subgroup of $G(K)$ with $m(\bP_{m})=m$.  The vector space $V_{m}:=V_{\bP_{m}}$ is
\begin{equation}\label{uV}
V_{m}=\Hom(M_{1},M_{0})\oplus\cdots\oplus\Hom(M_{\ell}, M_{\ell-1})\oplus\Sym^{2}(M_{\ell}).
\end{equation}
A better way to think of $V_{m}$ is the following. We consider the cyclic quiver
\begin{equation}\label{uquiver}
\xymatrix{ & M_{1}\ar[dl]^{\psi_{0}} & M_{2}\ar[l]^{\psi_{1}} & \cdots\ar[l] & M_{\ell}\ar[l]^{\psi_{\ell-1}} \\
M_{0}\ar[dr]^{\psi_{-1}}\\
& M_{-1}\ar[r]^{\psi_{-2}} & M_{-2}\ar[r] & \cdots\ar[r]^{\psi_{-\ell}} & M_{-\ell}\ar[uu]^{\psi_{\ell}} }
\end{equation}
There is an involution $\tau$ acting on such quivers: $\tau$ sends $\{\psi_{i}:M_{i+1}\to M_{i}\}$ to $\{\psi^{*}_{-i-1}: M_{i+1}=M_{-i-1}^{*}\to M_{-i}^{*}=M_{i}\}$ (indices understood modulo $2\ell+1$). Then $V_{m}$ is the space of $\tau$-invariant cyclic quivers.

When $m=2$, the cyclic quiver degenerates to a quiver with a single vertex $M=M_{0}$ and a self-adjoint operator $\psi$ on itself. 

\subsection{Stable locus}\label{ss:ustable}
The dual space $V^{*}_{m}$ is the space of $\tau$-invariant cyclic quivers of the shape \eqref{uquiver}, except that all arrows are reversed. Let $\phi_{i}:M_{i}\to M_{i+1}$ be the arrows. We think of $\phi_{\ell}:M_{\ell}\to M_{-\ell}$ as a quadratic form on $M_{\ell}$. An element $\phi=(\phi_{-\ell},\cdots, \phi_{\ell})\in V^{*}_{m}$ is stable under the $L_{m}$-action if and only if
\begin{itemize}
\item All maps $\phi_{i}$ have the maximal possible rank;
\item We have two quadratic forms on $M_{0}$: $q_{0}$ and the pullback of $\phi_{\ell}$ to $M_{0}$ via the map $\phi_{\ell-1}\cdots\phi_{0}:M_{0}\to M_{\ell}$. We denote the second quadratic form by $q'_{0}$. Then these quadratic forms are in general position: the pencil of quadrics spanned by $q_{0}$ and $q'_{0}$ is degenerate at $n$ distinct points on $\PP^{1}$. 
\end{itemize}
In particular, when $\dim M_{0}=d$ and $\phi$ is stable, then all $\phi_{i}$ are isomorphisms, and $\phi_{\ell}$ has at most a one-dimensional kernel. When $\dim M_{0}=d+1$ and $\phi$ is stable, then $\phi_{0}$ is surjective, $\phi_{1},\cdots, \phi_{\ell-1}$ are isomorphisms, and $\phi_{\ell}$ must be nondegenerate.

When $m=2$, $V^{*}_{2}$ is simply the space of quadratic forms on $M$. An element $\phi\in V^{*}_{2}(k)$ is stable under the $L_{2}=\SO(M,q)$-action exactly when the quadratic form $\phi$ is in general position with $q$ (note $\phi$ itself can be degenerate).

\subsection{The scheme $\frG_{\l}$}
According to \S\ref{ss:cal}, we need to describe the variety $\frG_{\l}$ for the minuscule $\lambda=(1,0,\cdots)$ (corresponding to the standard representation of $\dualG=\GL_{n}$) and the morphism $(f',f''):\frG_{\l}\to \tLab_m\times V_m$. 

Let $\calE=M\otimes\calO_{\tX}$ with Hermitian form (into $\calO_{\tX}$) given by a formula  similar to \eqref{herm}. There is an increasing filtration $F_{*}$ on the fiber of $\calE$ at $\infty$ defined by $F_{\leq i}=\sum_{j\leq i}M_{i}$. We have the tautological trivializations $\Gr^{F}_{i}\calE=M_{i}$. There is a decreasing filtration $F^{*}$ on the fiber of $\calE$ at $0$ defined by $F^{\geq i}=\sum_{j\geq i}M_{i}$. The triple $(\calE,F_{*}, F^{*})$ defines the open point of $\Bun_{G}(\wt\bP_{0},\bP_{\infty}^{+})$. 

The group ind-scheme $\frG$ is the group of unitary automorphisms of $(\calE|_{\tX-\{\pm1\}}, F_{*}, F^{*})$ that induces the identity on the associated graded of $F_{*}$. The locus $\frG_{\l}$ consists of those $g\in\frG\subset G(F')$ ($F'=k(t^{1/2})$) that have a simple pole at $t^{1/2}=1$ with residue of rank one. Such $g$ can be written uniquely as
\begin{equation*}
g=\frac{t^{1/2}}{t^{1/2}-1}A-\frac{1}{t^{1/2}-1}B
\end{equation*}
for $A,B\in\GL(M)$.

\begin{lemma}\label{l:upre}
\begin{enumerate}
\item The scheme $\frG_{\l}$ classifies pairs $(A,B)\in\Og(M,q)\times\Og(M,q)$ such that
\begin{itemize}
\item $(Ax,By)=(Bx,Ay)$ for all $x,y\in M$.
\item $A$ lies in the unipotent radical $U(F_{*})$ of the parabolic $P(F_{*})\subset\Og(M,q)$ preserving the filtration $F_{*}$ on $M$; $B$ lies in the parabolic $P(F^{*})\subset \Og(M,q)$ preserving the filtration $F^{*}$ on $M$.
\item $A-B$ has rank one. 
\end{itemize}

\item The map $(f',f''):\frG_{\l}\to\tLab_{m}\times V_{m}$ is given by
\begin{eqnarray*}
f'(A,B)=(\det(B_{0,0}), \cdots,\det(B_{\ell,\ell}));\\
f''(A,B)=(A_{0,1}, A_{1,2}, \cdots, A_{\ell-1,\ell}, -B_{\ell,-\ell})
\end{eqnarray*}
in the block presentation of $A,B$ under the decomposition \eqref{ud}.

\item Consider the morphism $j:\frG_{\l}\to\PP(M)$ sending $(A,B)\in\frG_{\l}$ to the line in $M$ that is the image of the rank one endomorphism $C:=I-A^{-1}B\in \End(M)$. Then $j$ is an open embedding.
\end{enumerate}
\end{lemma}
\begin{proof}
(1) The fact $g=\frac{t^{1/2}}{t^{1/2}-1}A-\frac{1}{t^{1/2}-1}B$ is unitary is equivalent to $A,B\in\Og(M,q)$ and $(Ax,By)=(Bx,Ay)$. The residue of $g$ at $t^{1/2}=1$ is $A-B$, which hence has rank one.

(2) The value of $g$ at $t^{1/2}=0$ is $B$, hence we have the formula for $f'(A,B)$. Expanding $g$ at  $t^{1/2}=\infty$ using the local uniformizer $t^{-1/2}$, we get $g=A+t^{-1/2}(A-B)+O(t^{-1})$, which gives the formula for $f''(A,B)$.

(3) The map $j$ is an open embedding since $U(F_{*})\times P(F^{*})$ is open in $\Og(M,q)$.
\end{proof}

For $i=0,\cdots, \ell$, let $q_{[-i,i]}$ be the restriction of the  quadratic form $q$ to $M_{-i}\oplus\cdots\oplus M_{i}$, and then extended to $M$ by zero on the rest of the direct summands (so that $q_{[-\ell,\ell]}=q$). 

\begin{prop}\label{p:ucal}
\begin{enumerate}
\item Under the embedding $j:\frG_{\l}\incl\PP(M)$, $\frG_{\l}$ is the complement of the divisors $q_{0}=0, q_{[-1,1]}=0,\cdots, q_{[-\ell+1,\ell-1]}=0$ and $q=0$. The map $(f',f''):\PP(M)\supset\frG_{\l}\to\tLab_{m}\times V_{m}$ is given by 
\begin{eqnarray}
\label{uf'} f'([v])=(-1, \frac{q_{0}(v)}{q_{[-1,1]}(v)},\frac{q_{[-1,1]}(v)}{q_{[-2,2]}(v)},\cdots,\frac{q_{[-\ell+2,\ell-2]}(v)}{q_{[-\ell+1,\ell-1]}(v)},\frac{q_{[-\ell+1,\ell-1]}(v)}{q(v)});\\
\label{uf''} f''([v])=(\frac{(-, v_{-1})}{q_{0}(v)}v_{0}, \frac{(-, v_{-2})}{q_{[-1,1]}(v)}v_{1}, \cdots, \frac{(-,v_{-\ell})}{q_{[-\ell+1,\ell-1]}(v)}v_{\ell-1}, \frac{(-,v_{\ell})}{q(v)}v_{\ell}).
\end{eqnarray}
\item The local system $\Kl^{\St}_{\dualG, \bP_{m}}(\chi)$ on $V_{m}^{*,\rs}$ attached to the unitary group $G$, the admissible parahoric $\bP_{m}$, the character $\chi=(\chi_{0}, \chi_{1},\cdots,\chi_{\ell})$ (where $\chi_{0}$ has order two) and the standard representation of the dual group $\dualG=\GL_{n}$ is the Fourier transform of the complex $f''_{!}f'^{*}\calL_{\chi}[n-1](\frac{n-1}{2})$.
\end{enumerate}
\end{prop}
\begin{proof}
(1) For $(A,B)\in\frG_{\l}$, let $C=I-A^{-1}B$ which has rank one. The condition $(Ax,By)=(Bx,Ay)$ implies $(Cx,y)=(x,Cy)$ (for all $x,y\in M$). This together with the fact that $I-C\in \Og(M,q)$ implies that $C(x)=\frac{(x,v)}{q(v)}v$ for some $v\in M$ with $q(v)=(v,v)/2\neq0$. In other words, $I-C=R_{[v]}$, the orthogonal reflection in the direction of $v$. Note that $j(A,B)=[v]\in\PP(M)$. 

We inductively find $A=A_{1}\cdots A_{\ell-1} A_{\ell}$ such that $AR_{[v]}\in \overline{P}$. Here $A_{i}$ is the identity on $M_{j}$ for all $j\neq i$ and sends $x_{i}\in M_{i}$ to $x_{i}\mod F_{<i}$. It is easy to check that
\begin{eqnarray*}
A_{\ell}(x_{\ell})&=&x_{\ell}+\frac{(x_{\ell}, v_{-\ell})}{q_{[-\ell+1,\ell-1]}(v)}v_{\leq \ell-1};\\
A_{\ell-1}(x_{\ell-1})&=&x_{\ell-1}+\frac{(x_{\ell-1}, v_{-\ell+1})}{q_{[-\ell+2,\ell-2]}(v)}v_{\leq \ell-2};\\
&&\cdots;\\
A_{1}(x_{1})&=&x_{1}+\frac{(x_{1}, v_{-1})}{q_{0}(v)}v_{\leq0}.
\end{eqnarray*}
Here we write $v_{\leq i}$ for the projection of $v$ to the factor $\oplus_{j\leq i}M_{j}$ of $M$.
Taking the blocks $A_{i,i+1}$ we get the formula for $f''([v])$ except for the last entry. The last entry of $f''([v])$ is the negative of the lower-left corner block $B_{\ell, -\ell}$, which is the same as the corner block of $R_{[v]}$. The formula we get is exactly \eqref{uf''}. The action of $B=AR_{[v]}$ on the associated graded $M_{i}=\Gr^{i}_{F^{*}}M$ is given by
\begin{equation*}
B_{i,i}(x_{i})=x_{i}-\frac{(x_{i},v_{-i})}{q_{[-i,i]}(v)}v_{i}.
\end{equation*}
In particular, $B_{0,0}$ is the reflection $R_{v_{0}}$ on $M_{0}$. Taking determinants we get the formula for $f'([v])$ as in \eqref{uf'}.

(2) For the coweight $\l$, we have $\IC_{\l}=\Ql[n-1](\frac{n-1}{2})$. Therefore (2) follows from Proposition \ref{p:Four}.
\end{proof}

A direct consequence of Proposition \ref{p:ucal} is the following.
\begin{cor}\label{c:uKl} Let $\phi=(\phi_{-\ell},\cdots,\phi_{\ell})\in V^{*,\st}_{m}(k)$ be a stable functional. Recall $\frG_{\l}\cong\PP(M)-\cup_{i=0}^{\ell}Q(q_{[-i,i]})$, the complement of the quadrics defined by the $q_{[-i,i]}$. Let $f_{\phi}: \tpline\times \frG_{\l}\to \AA^{1}$ be given by
\begin{equation*}
f_{\phi}(x,[v])=\sum_{i=0}^{\ell-1}\frac{(\phi_{i}v_{i}, v_{-i-1})}{q_{[-i,i]}(v)}+\frac{\phi_{\ell}(v_{\ell})}{q(v)}x.
\end{equation*}
Note that $\phi_{\ell}(-)$ is a quadratic form. Let $\pi:\tpline\times\frG_{\l}\to\tpline$ be the projection.
Then we have an isomorphism over $\tpline$
\begin{equation}\label{uStphi}
\Kl^{\St}_{\dualG,\bP_{m}}(1,\phi)\cong\pi_{!}f_{\phi}^{*}\AS_{\psi}[n-1]\left(\frac{n-1}{2}\right).
\end{equation}
\end{cor}
\begin{proof}
We identify $\tpline$ with the torus $\grot$. The action of $\grot$ on $V^{*}_{m}$ is by scaling the last component $\Sym^{2}(M_{\ell}^{*})$. Let $a_{\phi}: \grot\to V^{*,\st}_{m}$ be given by the $\grot$-orbit of $\phi$. It takes the form $x\mapsto(\phi_{0},\cdots, \phi_{\ell-1}, x\phi_{\ell})$. By Corollary \ref{c:Klphi} and Proposition \ref{p:ucal}(2),
\begin{equation*}
\Kl^{\St}_{\dualG,\bP_{m}}(\chi,\phi)=a_{\phi}^{*}\Kl^{\St}_{\dualG, \bP_{m}}(\chi)\cong a_{\phi}^{*}\Four_{\psi}(f''_{!}\Ql)[n-1]\left(\frac{n-1}{2}\right).
\end{equation*}
By proper base change, the last term is isomorphic to $\pi_{!}b_{\phi}^{*}\AS_{\psi}[n-1]$ where $b_{\phi}:\grot\times\frG_{\l}\to \AA^{1}$ is given by $(x,[v])\mapsto x\cdot_{\rot}f''([v])$. The explicit formula of $f''$ in Proposition \ref{p:ucal}(1) shows that $b_{\phi}=f_{\phi}$. Therefore \eqref{uStphi} holds.
\end{proof}

When $m=2$, we may describe the local system $\Kl^{\St}_{\dualG, \bP}(1, \phi)$ more explicitly. 

\begin{prop}\label{p:um2} Let $\phi\in V^{*,\st}_{2}(k)$ such that the pencil of quadrics spanned by $q$ and $\phi$ is degenerate exactly at $\phi-\l_{i}q$ for distinct elements $\lambda_1,\cdots,\lambda_n\in\bark$.
\begin{enumerate}
\item Consider the morphism
\begin{eqnarray*}
f_\phi:\PP(M)-Q(q) & \to & \AA^1 \\
{[v]} &\mapsto & \frac{\phi(v)}{q(v)}.
\end{eqnarray*}
The local system $\Kl^{\St}_{\dualG, 2}(1, \phi)$ over $\tpline\cong\grot$ is the restriction of the Fourier transform of $f_{\phi,!}\Ql[n-1](\frac{n-1}{2})$ from $\AA^{1}$ to $\grot$.
\item When $n$ is odd, $\Kl^{\St}_{\dualG, 2}(1, \phi)$ over $\Gm\otimes_{k}\bark$ is isomorphic to $\oplus_{i=1}^{n}\AS_{\psi_{i}}|_{\grot}$, where $\psi_i(t)=\psi(\lambda_it)$.
\item When $n$ is even, then $\Kl^{\St}_{\dualG, 2}(1, \phi)$ over $\Gm\otimes_{k}\bark$ is isomorphic to $\Four_{\psi}(\calL_{\phi}[1])|_{\grot}$, where $\calL_{\phi}$ is the middle extension of the rank one local system on $\AA^1-\{\lambda_1,\cdots,\lambda_n\}$ with monodromy equal to $-1$ around each puncture $\lambda_i$.
\end{enumerate}
\end{prop}
\begin{proof} We work over the base field $\bark$ without changing notation.
(1) One checks that the map $f_{\phi}$ is the composition $\PP(M)-Q(q)\xrightarrow{f''}V_{2}\xrightarrow{\phi}\AA^{1}$. Therefore (1) is a direct consequence of Proposition \ref{p:ucal}(2). 

(2)(3) The morphism $f_\phi$ can be compactified into a pencil of quadrics spanned by $\phi$ and $q$:\begin{equation*}
\tilf_\phi:\Bl_Z\PP(M)\to\PP^1
\end{equation*}
where $Z=Q(\phi)\cap Q(q)$ is the base locus of this pencil and $\Bl_Z\PP(M)$ is the blow-up of $\PP(M)$ along $Z$. The fiber of $\tilf_{\phi}$ over $[x,y]\in\PP^{1}$ is the quadric defined by $y\phi-xq=0$. Let $j:\PP(M)-Z\hookrightarrow\Bl_Z\PP(M)$ be the open immersion and $i:\PP^1\times Z\hookrightarrow\Bl_Z\PP(M)$ be the closed immersion of the exceptional divisor. Then $f_{\phi,!}\Ql$ is the restriction of $(\tilf_{\phi}j)_!\Ql$ to $\AA^1$. We have an distinguished triangle in $D^b_c(\PP^1)$:
\begin{equation*}
(\tilf_{\phi}j)_!\Ql\to\tilf_{\phi,*}\Ql\to \cohog{*}{Z}\otimes\const{\PP^1}\to(\tilf_{\phi}j)_!\Ql[1].
\end{equation*}
Therefore, $\tilf_{\phi,*}\Ql|_{\AA^1}$ and $f_{\phi,!}\Ql$ differ by a constant sheaf, hence their Fourier transforms are the same over $\grot=\tpline$. 

Since $\phi$ is a stable point, the degenerate quadrics $\phi-\l_{i}q$ in the pencil are projective cones over $(n-3)$-dimensional quadrics. All quadrics (degenerate or not) have vanishing odd degree cohomology. Let $c_1$ be the Chern class of the line bundle $\calO_{\PP(M)}(1)$, viewed as a morphism $\Ql(-1)\to\bR^2\tilf_{\phi,*}\Ql$. The $i$-th power of $c_1$ gives a morphism
\begin{equation*}
c_1^i:\Ql(-i)\to\bR^{2i}\tilf_{\phi,*}\Ql.
\end{equation*}

When $n$ is odd, $c_1^i$ is an isomorphism for $i\neq\frac{n-1}{2}$. When $i=\frac{n-1}{2}$, $c_1^i$ is an injection of sheaves with cokernel a sum of skyscraper sheaves supported at the critical values of $\tilf_\phi$, i.e., $\{\lambda_1,\cdots,\lambda_n\}$. Therefore, $f_{\phi,!}\Ql[n-1]$ differs from the sum of these skyscraper sheaves by constant sheaves. This implies (2).

When $n$ is even, $c_1^i$ is an isomorphism for $i\neq\frac{n-2}{2}$. When $i=\frac{n-2}{2}$, $c_1^i$ is an injection with cokernel $\calL_{\phi}$ being the extension by zero of a rank one local system on $\AA^{1}-\{\lambda_1,\cdots,\lambda_n\}$. By Picard-Lefschetz theory, $\calL_{\phi}$ has monodromy $-1$ around $\{\lambda_1,\cdots,\lambda_n\}$, hence $\calL_{\phi}[1]$ is an irreducible perverse sheaf on $\AA^{1}$. Now $f_{\phi,!}\Ql[n-1]$ differs from $\calL_{\phi}[1]$ by constant sheaves, which implies (3).
\end{proof}

Finally we calculate the Euler characteristic of $\Kl^{\St}_{\dualG,\bP_{m}}(1,\phi)$, which then gives evidence for the conjectural description of its Swan conductor at $\infty$ in \S\ref{ss:wildL}.
\begin{prop}\label{p:uchi} We have
\begin{equation}\label{uchi}
-\chi_{c}(\tpline, \Kl^{\St}_{\dualG,\bP_{m}}(\chi,\phi))=\begin{cases}d & \phi_{\ell}\textup{ nondegenerate;}\\ d-1 & \phi_{\ell}\textup{ degenerate.}\end{cases}
\end{equation}
\end{prop}
\begin{proof} We shall work over $\bark$ and ignore all Tate twists.

By the same argument as in \cite[Proposition 10.1]{Katz}, the Swan conductor of $\Kl^{\St}_{\dualG,\bP_{m}}(\chi,\phi)$ at $\infty$, hence the Euler characteristic of $\Kl^{\St}_{\dualG,\bP_{m}}(\chi,\phi)$ does not depend on $\chi$. Therefore we assume $\chi=1$. 

Let $Q_{i}\subset\PP(M)$ be the quadric defined by $q_{[-i,i]}=0$. Let $U_{i}=\PP(\oplus_{j=-i}^{i}M_{j})-\cup_{j=0}^{i}Q_{i}$. In particular $U_{\ell}=\frG_{\l}$. We also define $W_{i}\subset U_{i}$ to be the quadric defined by $\phi_{\ell}(\phi_{\ell-1}\cdots\phi_{i}v_{i})=0$.

Let $f_{i}: U_{\ell}\to \AA^{1}$ be the function $[v]\mapsto \frac{(\phi_{i}v_{i}, v_{-i-1})}{q_{[-i,i]}([v])}$. This function only depends on the coordinates $v_{-i-1},\cdots, v_{0}, \cdots, v_{i}$. Let $f_{\leq i}=f_{0}+f_{1}+\cdots+f_{i}$, and $f_{\leq-1}:=0$.

Consider the projection $\pi_{2}: \grot\times\frG_{\l}\to \frG_{\l}=\PP(M)-\cup Q_{i}$.  The stalk of $\pi_{2,!} f_{\phi}^{*}\AS_{\psi}$ over $[v]$ is $f_{\leq\ell-1}^{*}\AS_{\psi}\otimes\cohog{*}{\grot, T_{f_{\ell}([v])}^{*}\AS_{\psi}}$, where $T_{f_{\ell}([v])}$ is the multiplication by $f_{\ell([v])}$ map $\Gm\to \AA^{1}$. We have $\cohog{*}{\grot, T_{f_{\ell}([v])}^{*}\AS_{\psi}}=\cohoc{*}{\Gm}$ if $f_{\ell}([v])=0$ and $\Ql[-1]$ if $f_{\ell}([v])\neq0$. Therefore
\begin{equation}\label{chi1}
(-1)^{n-1}\chi_{c}(\tpline, \Kl^{\St}_{\dualG,\bP_{m}}(1,\phi))=\chi_{c}(\pi_{2,!} f_{\phi}^{*}\AS_{\psi})=-\chi_{c}(U_{\ell}, f_{\leq\ell-1}^{*}\AS_{\psi})+\chi_{c}(W_{\ell}, f_{\leq\ell-1}^{*}\AS_{\psi}).
\end{equation}
We then show that for $0\leq i\leq \ell-1$, we have
\begin{eqnarray}
\label{Ui} \chi_{c}(U_{i+1}, f_{\leq i}^{*}\AS_{\psi})=\chi_{c}(U_{i}, f_{\leq i-1}^{*}\AS_{\psi});\\
\label{Wi} \chi_{c}(W_{i+1}, f_{\leq i}^{*}\AS_{\psi})=\chi_{c}(W_{i}, f_{\leq i-1}^{*}\AS_{\psi}).
\end{eqnarray}

For \eqref{Ui}, let $U'_{i}=\PP(M_{-i-1}\oplus M_{i})-\cup_{j=0}^{i}Q_{i}$. Consider the projection $p: U_{i+1}\to U'_{i}$ by forgetting the $M_{i+1}$ component. Then we have  $\chi_{c}(U_{i+1}, f_{\leq i}^{*}\AS_{\psi})=\chi_{c}(U'_{i}, f_{\leq i}^{*}\AS_{\psi}\otimes p_{!}\Ql)$. Fix $[v']=[v_{-i-1},\cdots, v_{i}]\in U_{i}$ and let $q_{i}:=q_{i}([v'])$ (which is independent of $v_{-i-1}$). The fiber of $p$ over $[v']$ is $M_{i+1}$ with the affine hyperplane $(v_{i+1}, v_{-i-1})+q_{i}=0$ removed. When $v_{-i-1}\neq0$, we have $\cohoc{*}{p^{-1}([v'])}\cong \cohoc{*}{\Gm}[-2d+2]$; when $v_{-i-1}=0$,  we have $\cohoc{*}{p^{-1}([v'])}\cong \Ql[-2d]$. Since $U_{i}$ can be identified with the subscheme of $U'_{i}$ where $v_{-i-1}=0$, we have $\chi_{c}(U_{i+1}, f_{\leq i}^{*}\AS_{\psi})=\chi_{c}(U'_{i}, f_{\leq i}^{*}\AS_{\psi}\otimes p_{!}\Ql)=\chi_{c}(U_{i}, f_{\leq i-1}^{*}\AS_{\psi})$ (the function $f_{\leq i}$ changes to $f_{\leq i-1}$ because $v_{-i-1}=0$).

For \eqref{Wi}, let $p:W_{i+1}\to W_{i}$ is the projection. We have $p_{!}f_{\leq i}^{*}\AS_{\psi}=f_{\leq i-1}^{*}\AS_{\psi}\otimes p_{!}f^{*}_{i}\AS_{\psi}$, and we need to compare $p_{!}f^{*}_{i}\AS_{\psi}$ with the constant sheaf on $W_{i}$. We decompose $p$ into two steps $W_{i+1}\xrightarrow{p_{1}}W''_{i}\xrightarrow{p_{2}}U_{i}$, where $W''_{i}\subset\PP(M_{-i}\oplus M_{i+1})-\cup_{j=0}^{i}Q_{i}$ is cut out the same the quadric $\phi_{\ell}(\phi_{\ell-1}\cdots\phi_{i+1}v_{i+1})=0$. Fix $v''=[v_{-i},\cdots, v_{i+1}]\in W''_{i}$ and let $q_{i}:=q_{i}([v''])$. The fiber $p_{1}^{-1}([v''])=\{v_{-i-1}\in M_{-i-1}|(v_{-i-1}, v_{i+1})+q_{i}\neq0\}$. The function $f_{i}$ along the fiber $p_{1}^{-1}([v''])$ is a linear function in $v_{-i-1}$ given by $f_{i}(v_{-i-1})=\frac{(\phi_{i}v_{i}, v_{-i-1})}{q_{i}}$. Therefore the stalk of $p_{1,!}f_{i}^{*}\AS_{\psi}$ at $[v'']$, which is $\cohoc{*}{p_{1}^{-1}([v'']), f_{i}^{*}\AS_{\psi}}$, vanishes whenever $v_{i+1}$ is not parallel to $\phi_{i}v_{i}$ (or one of them is zero while the other one is not). When $\phi_{i}v_{i}\neq0$, the stalk of $p_{2,!}p_{1,!}f_{i}^{*}\AS_{\psi}$ can be calculated only along those nonzero $v_{i+1}$ parallel to $\phi_{i}v_{i}$. Choose an isomorphism $M_{-i-1}\cong\AA^{d}$ such that the first coordinate is given by the functional $(\phi_{i}v_{i},-)/q_{i}$. Then the stalk of $p_{2,!}p_{1,!}f_{i}^{*}\AS_{\psi}$ is $\cohoc{*}{\Gm\times\AA^{1}-\Delta(\Gm), \textup{pr}_{2}^{*}\AS_{\psi}}[-2d+2]\cong\Ql[-2d+2]$ (here $\Delta(\Gm)\subset\Gm\times\AA^{1}$ is the diagonal). For $\phi_{i}v_{i}=0$,  the stalk of $p_{2,!}p_{1,!}f_{i}^{*}\AS_{\psi}$ can be calculated only along $v_{i+1}=0$, hence equal to $\cohoc{*}{M_{-i-1}}\cong\Ql[-2d]$. Putting together, we conclude that $p_{!}f^{*}_{i}\AS_{\psi}$ and the constant sheaf $\Ql$ are the same in the Grothendieck group of $D_{c}^{b}(W_{i})$. Therefore $\chi_{c}(W_{i+1}, f_{\leq i}^{*}\AS_{\psi})=\chi_{c}(W_{i}, f_{\leq i-1}^{*}\AS_{\psi}\otimes p_{!}f^{*}_{i}\AS_{\psi})=\chi_{c}(W_{i}, f_{\leq i-1}^{*}\AS_{\psi})$.

Combining \eqref{Ui}, \eqref{Wi} and \eqref{chi1}, we see that 
\begin{equation*}
(-1)^{n-1}\chi_{c}(\tpline, \Kl^{\St}_{\dualG,\bP_{m}}(1,\phi))=-\chi_{c}(U_{0})+\chi_{c}(W_{0}).
\end{equation*}
Now $U_{0}=\PP(M_{0})-Q_{0}$ and $W_{0}=Q'_{0}-Q_{0}$, where $Q'_{0}$ is the quadric in $\PP(M_{0})$ defined by $\phi_{\ell}(\phi_{\ell-1}\cdots\phi_{0}(v_{0}))=0$. Hence 
\begin{equation}\label{chiPQ}
(-1)^{n-1}\chi_{c}(\tpline, \Kl^{\St}_{\dualG,\bP_{m}}(1,\phi))=-\chi_{c}(\PP(M_{0}))+\chi_{c}(Q_{0})+\chi_{c}(Q'_{0})-\chi_{c}(Q_{0}\cap Q'_{0}). 
\end{equation}
Since $\phi$ is stable, all $\phi_{i}$ are surjective for $0\leq i\leq \ell-1$, $Q'_{0}$ is either minimally degenerate or nondegenerate, and $Q_{0}\cap Q'_{0}$ is always smooth (and equal to the intersection of two smooth quadrics). 

In the following table we list the dimensions of the primitive cohomology of $Q_{0}, Q'_{0}$ (always in even degree) and $Q_{0}\cap Q'_{0}$ (in the middle degree, which is $\dim M_{0}-3$) in each case according to whether $\dim M_{0}=d$ or $d+1$, and the parity of $d$. Using this table and \eqref{chiPQ} one easily deduces \eqref{uchi}.

\begin{tabular}{|c|c|c|c|c|}
\hline
$\dim M_{0}$ & Parity of $\dim M_{0}$ & $Q_{0}$ & $Q'_{0}$ & $Q_{0}\cap Q'_{0}$ \\   \hline
$d$ & even &  1 & 1 (nondeg) or 0 (degen) & $d-2$ \\   \hline
$d$ & odd &  0 & 0 (nondeg) or 1 (degen) & $d$ \\   \hline
$d+1$ & even & 1 & 0  & $d-1$ \\   \hline
$d+1$ & odd & 0 & 1 & $d+1$ \\   \hline
\end{tabular}
\end{proof}

\section{Examples: Symplectic groups}\label{s:sp}

\subsection{Linear algebra}
Let $(M,\omega)$ be a symplectic vector space of dimension $2n$ over $k$ and let $\GG=\Sp(M,\omega)$. Extend $\omega$ linearly to a symplectic form on $M\otimes K$ and let $G=\Sp(M\otimes K,\omega)$. Regular elliptic numbers $m$ of $\WW$ in this case are in bijection with divisors $d|n$. We have $m=2n/d$ and set $\ell=n/d$. Fix a decomposition
\begin{equation}\label{spd}
M=M_{1}\oplus M_{2}\oplus\cdots\oplus M_{\ell}\oplus M_{\ell+1}\oplus\cdots\oplus M_{m}
\end{equation}
such that $\dim_{k}M_{i}=d$ and $\omega(M_{i}, M_{j})=0$ unless $i+j=m+1$. Then $M_{j}$ can be identified with $M^{*}_{m+1-j}$ using the symplectic form $\omega$.

Let $\bP_{m}\subset G(K)$ be the stabilizer of the chain of lattices $\L_{m}\supset\L_{m-1}\supset\cdots\supset\L_{1}$, where 
\begin{equation*}
\L_{i}=\sum_{1\leq j\leq i}M_{j}\otimes\calO_{K}+\sum_{i<j\leq m}M_{j}\otimes\varpi\calO_{K}
\end{equation*}
and $\varpi\in\calO_{K}$ is a uniformizer. This is an admissible parahoric subgroup of $G(K)$ with $m(\bP_{m})=m$. Its Levi quotient is $L_{m}\cong\prod_{i=1}^{\ell}\GL(M_{i})$ where the $i$-th factor acts on $M_{i}$ by the standard representation and on $M_{m+1-i}=M^{*}_{i}$ by the dual of the standard representation. We have $\tL_{m}=L_{m}$ and
\begin{equation}\label{spL}
\tLab_{m}\cong\prod_{i=1}^{\ell}\Gm
\end{equation}
given by the determinants of the $\GL$-factors. The vector space $V_{m}:=V_{\bP_{m}}$ is
\begin{equation}\label{spV}
V_{m}=\Sym^{2}(M^{*}_{1})\oplus\Hom(M_{2}, M_{1})\oplus\cdots\oplus\Hom(M_{\ell}, M_{\ell-1})\oplus\Sym^{2}(M_{\ell}).
\end{equation}
We may arrange $M_{1},\cdots, M_{m}$ into a cyclic quiver
\begin{equation}\label{spquiver}
\xymatrix{M_{1}\ar[d]^{\psi_{m}} & M_{2}\ar[l]^{\psi_{1}} & \cdots\ar[l] & M_{\ell} \ar[l]^{\psi_{\ell-1}}\\
M_{m}\ar[r]^{\psi_{m-1}} & M_{m-1}\ar[r] & \cdots\ar[r]^{\psi_{\ell+1}} & M_{\ell+1}\ar[u]^{\psi_{\ell}}}
\end{equation}
There is an involution $\tau$ on the space of such quivers sending $\{\psi_{i}:M_{i+1}\to M_{i}\}$ to $\{-\psi^{*}_{m-i}:M_{i+1}=M_{m-i}^{*}\to M_{m+1-i}^{*}=M_{i}\}$. Then $V_{m}$ is the set of $\tau$-invariant cyclic quivers of the above shape.

When $m=2$, the cyclic quiver degenerates to a pair of self-adjoint maps
\begin{equation*}
\xymatrix{M_{1}\ar@<1ex>[r]^{\psi_{2}} & M_{2}\ar@<1ex>[l]^{\psi_{1}}}
\end{equation*}

\subsection{Stable locus} The dual space $V^{*}_{m}$ is the space of $\tau$-invariant cyclic quivers similar to \eqref{spquiver}, except that all the arrows are reversed.  Let $\phi_{i}: M_{i}\to M_{i+1}$ be the arrows. We think of $\phi_{m}$ and $\phi_{\ell}$ as quadratic forms on $M_{m}$ and $M_{\ell}$ respectively. Then $\phi=(\phi_{1}, \cdots, \phi_{m})\in V^{*}_{m}$ is stable if and only if
\begin{itemize}
\item All the maps $\phi_{i}$ are isomorphisms;
\item We have two quadratic forms on $M_{m}$: $\phi_{m}$ and the transport of $\phi_{\ell}$ to $M_{m}$ using the isomorphism $\phi_{\ell-1}\cdots\phi_{1}\phi_{m}:M_{m}\isom M_{\ell}$. They are in general position in the same sense as explained in \S\ref{ss:ustable}.
\end{itemize}

\subsection{The scheme $\frG_{\leq\l}$}
Let $\l\in\xcoch(T)$ be the dominant short coroot. The corresponding representation of $\dualG=\SO_{2n+1}$ is the standard representation. We shall describe $\frG_{\leq\l}$ and the map $(f',f''):\frG_{\leq\l}\to \tLab_{m}\times V_{m}$.

Let $\calE=M\otimes\calO_{X}$ be the trivial vector bundle of rank $2n$ over $X$ with a symplectic form (into $\calO_{X}$) given by $\omega$. Define an increasing filtration of the fiber of $\calE$ at $\infty$ by $F_{\leq i}\calE_{\infty}=\oplus_{j=1}^{i}M_{j}$. Define a decreasing filtration of the fiber of $\calE$ at $0$ by $F^{\geq i}\calE_{0}=\oplus_{j=i}^{m}M_{j}$. The group ind-scheme $\frG$ is the group of symplectic automorphisms of $\calE|_{X-\{1\}}$ preserving the filtrations $F_{*}, F^{*}$ and acts by identity on the associated graded of $F_{*}$. The subscheme $\frG_{\leq\l}$ consists of those $g\in\frG \subset G(F)$ whose entries have at most simple poles at $t=1$, and $\Res_{t=1}g$ has rank at most one. Therefore any element in $\frG$ can be written uniquely as
\begin{equation*}
g=\frac{t}{t-1}A-\frac{1}{t-1}B
\end{equation*}
for $A,B\in\GL(M)$.

\begin{lemma}
\begin{enumerate}
\item The scheme $\frG_{\leq\l}$ classifies pairs $(A,B)\in\Sp(M,\omega)\times\Sp(M,\omega)$ satisfying
\begin{itemize}
\item $A$ lies in the unipotent radical $U(F_{*})$ of the parabolic $P(F_{*})\subset\Sp(M,\omega)$ preserving the filtration $F_{*}$ of $M$; $B$ lies in the opposite parabolic $P(F^{*})\subset\Sp(M,\omega)$ preserving the filtration $F^{*}$ of $M$.
\item $A-B$ has rank at most one. Moreover, $C=I-A^{-1}B\in\End(M)$ satisfies $\omega(Cx,y)+\omega(x,Cy)=0$ for all $x,y\in M$.
\end{itemize}

\item The map $(f',f''):\frG_{\leq\l}\to \tLab_{m}\times V_{m}$ is given by
\begin{eqnarray*}
f'(A,B)&=&(\det(B_{1,1}), \cdots, \det(B_{\ell,\ell}));\\
f''(A,B)&=&(-B_{m,1}, A_{1,2}, \cdots, A_{\ell,\ell+1}) 
\end{eqnarray*}
in the block presentation of $A,B$ under the decomposition \eqref{spd}. 

\item Let $\Sym^{2}(M)_{\leq1}\subset\Sym^{2}(M)$ be the subscheme of symmetric pure 2-tensors. Then $\Sym^{2}(M)_{\leq 1}$ can be identified with the scheme of endomorphisms $D$ of $M$ satisfying $\omega(Dx,y)+\omega(x,Dy)=0$ and of rank at most one).  Then the morphism $j:\frG_{\leq\l}\to \Sym^{2}(M)_{\leq1}$ sending $(A,B)\in\frG_{\leq\l}$ to $C=I-A^{-1}B$ is an open embedding.
\end{enumerate}
\end{lemma}
\begin{proof}
(1) The fact that $g=\frac{t}{t-1}A-\frac{1}{t-1}B$ preserves the symplectic form $\omega$ is equivalent to $A,B\in \Sp(M,\omega)$ and $\omega(Ax,By)+\omega(Bx,Ay)=2\omega(x,y)$ for all $x,y\in M$. The last condition is further equivalent to $\omega(Cx,y)+\omega(x,Cy)=0$ for all $x,y\in M$. Note that $g(0)=B, g(\infty)=A$ and $\Res_{t=1}g=A-B$, hence the rest of the conditions follows.

(2) is proved in the same way as Lemma \ref{l:upre}(2).

(3) Since $D$ has rank at most one, we may write it as $D(x)=\omega(x,u)v$ for some $u,v\in M$. The condition $\omega(Dx,y)+\omega(x,Dy)=0$ for all $x,y\in M$ implies that $u$ and $v$ are parallel vectors. Hence $u\cdot v\in\Sym^{2}(M)_{\leq1}$. The fact that $j$ is an open embedding follows from the fact that $U(F_{*})\times P(F^{*})\incl \Sp(M,\omega)$ is an open embedding.
\end{proof}

For $u\cdot v\in \Sym^{2}(M)_{\leq 1}$, write $u=(u_{1},\cdots, u_{m})$ and $v=(v_{1},\cdots, v_{m})$ with  $u_{i}, v_{i}\in M_{i}$. Then define
\begin{equation}\label{gammai}
\gamma_{i}(u\cdot v):=\omega(v_{m+1-i}, u_{i}).
\end{equation}
This is independent of the choice of $u,v$ expressing $u\cdot v$, and therefore defines a regular function on $\Sym^{2}(M)_{\leq 1}$. Note that $\gamma_{i}+\gamma_{m+1-i}=0$ since $u$ and $v$ are parallel.

\begin{prop}\label{p:spcal} 
\begin{enumerate}
\item Under the open embedding  $j:\frG_{\leq\l}\incl\Sym^{2}(M)_{\leq1}$, $\frG_{\leq\l}$ is the complement of the union of the divisors $\gamma_{1}+\cdots+\gamma_{i}=1$ for $i=1,\cdots, \ell$. The morphism $(f',f''): \Sym^{2}(M)_{\leq1}\supset \frG_{\leq\l}\to \tLab_{m}\times V_{m}$ is given by
\begin{eqnarray}
\label{spf'} f'(u\cdot v)&=&(\frac{1}{1-\gamma_{1}}, \frac{1-\gamma_{1}}{1-\gamma_{1}-\gamma_{2}},\cdots, \frac{1-\gamma_{1}-\cdots-\gamma_{\ell-1}}{1-\gamma_{1}-\cdots-\gamma_{\ell}})\\
\label{spf''} f''(u\cdot v)&=&(\omega(-,u_{m})v_{m}, \frac{\omega(-,u_{m-1})}{1-\gamma_{1}}v_{1}, \frac{\omega(-,u_{m-2})}{1-\gamma_{1}-\gamma_{2}}v_{2}, \cdots,\frac{\omega(-,u_{\ell})}{1-\gamma_{1}-\cdots-\gamma_{\ell}}v_{\ell}).
\end{eqnarray}
Here we abbreviate $\gamma_{i}(u\cdot v)$ by $\gamma_{i}$.
\item The local system $\Kl^{\St}_{\dualG, \bP_{m}}(\chi)$ attached to the standard representation of $\dualG=\SO_{2n+1}$, the admissible parahoric subgroup  $\bP_{m}$ and the multiplicative character $\chi=(\chi_{1},\cdots,\chi_{\ell})$ is given by (the restriction to $V^{*,\st}_{m}$ of) the Fourier transform of the complex $f''_{!}f'^{*}\calL_{\chi}[2n-1](\frac{2n-1}{2})$.
\end{enumerate}
\end{prop}
\begin{proof}
(1) Write $C=I-A^{-1}B$ as $\omega(-,u)v$ for parellel vectors $u,v\in M$. We inductively construct $A=A_{2}\cdots A_{m}\in U(F_{*})$ such that $B=A(I-C)\in P(F^{*})$, and that $A_{i}$ is the identity on $M_{j}, j\neq i$ and $A_{i}(x_{i})\in x_{i}+F_{\leq i-1}$ for any $x_{i}\in M_{i}$.  We may inductively determine
\begin{eqnarray*}
A_{m}x_{m}&=&x_{m}+\frac{\omega(x_{m},u_{1})}{1-\gamma_{1}}v_{\leq m-1};\\
A_{m-1}x_{m-1}&=&x_{m-1}+\frac{\omega(x_{m-1},u_{2})}{1-\gamma_{1}-\gamma_{2}}v_{\leq m-2};\\
&&\cdots\\
A_{2}x_{2}&=&x_{2}+\frac{\omega(x_{2},u_{m-1})}{1-\gamma_{1}-\cdots-\gamma_{m-1}}v_{1}.
\end{eqnarray*}
Here we write $v_{\leq i}$ for the projection of $v$ to the direct factor $\oplus_{j\leq i}M_{j}$ of $M$.
Therefore the entries of $A$ corresponding to $\Hom(M_{i+1}, M_{i})$  ($1\leq i\leq \ell$) takes the form
\begin{equation*}
A_{i, i+1}=\frac{\omega(-,u_{m-i})}{1-\gamma_{1}-\cdots-\gamma_{m-i}}v_{i}\in\Hom(M_{i+1}, M_{i}).
\end{equation*}
Using $\gamma_{i}=-\gamma_{m+1-i}$ to simplify the denominators, we get the formula given in \eqref{spf'} except for the first entry.
The corner block of $B=A(I-C)$ corresponding to $\Hom(M_{1}, M_{m})$ is the same as the corner block of $I-C$, which takes the form $-\omega(-,u_{m})v_{m}$, whose negative gives the first entry of the formula \eqref{spf'}.

The matrix $B=A(I-C)$ has block diagonal entries
\begin{equation*}
B_{i,i}=\id-\frac{\omega(-,u_{m+1-i})}{1-\gamma_{1}-\cdots-\gamma_{m-i}}v_{i}\in\GL(M_{i}).
\end{equation*}
Taking determinants we get 
\begin{equation*}
\det(B_{i,i})=\frac{1-\gamma_{1}-\cdots-\gamma_{m+1-i}}{1-\gamma_{1}-\cdots-\gamma_{m-i}}.
\end{equation*}
Using that $\gamma_{i}=-\gamma_{m+1-i}$ we get the formula \eqref{spf''}.

(2) We only need to apply Proposition \ref{p:Four} to the standard representation $V_{\l}$ of $\SO_{2n+1}$; note that $\IC_{\l}=\Ql[2n-1](\frac{2n-1}{2})$ in this case. 
\end{proof}

Similar to Corollary \ref{c:uKl}, we have
\begin{cor}\label{c:spKl}
Let $\phi=(\phi_{1}, \cdots, \phi_{m})\in V^{*,\st}_{m}(k)$ be a stable functional. Recall that $\frG_{\leq\l}$ in this case is $\Sym^{2}(M)_{\leq1}-\cup_{i=1}^{\ell}\Gamma_{i}$ where the divisor $\Gamma_{i}$ is defined by the equation $\gamma_{1}+\cdots+\gamma_{i}=1$ for functions $\gamma_{i}$ in \eqref{gammai}. Let $f_{\phi}: \pline\times \frG_{\leq\l}\to \AA^{1}$ be given by
\begin{equation*}
f_{\phi}(x,u\cdot v)=\omega(\phi_{m}v_{m}, u_{m})x+\sum_{i=1}^{\ell}\frac{\omega(\phi_{i}v_{i}, u_{m-i})}{1-\gamma_{1}(u\cdot v)-\cdots-\gamma_{i}(u\cdot v)}.
\end{equation*}
Let $\pi:\pline\times\frG_{\leq\l}\to\pline$ be the projection.
Then we have an isomorphism over $\pline$
\begin{equation*}
\Kl^{\St}_{\dualG,\bP_{m}}(1,\phi)\cong\pi_{!}f_{\phi}^{*}\AS_{\psi}[2n-1]\left(\frac{2n-1}{2}\right).
\end{equation*}
\end{cor}

\section{Examples: Split and quasi-split orthogonal groups}\label{s:o}
\subsection{Linear algebra}\label{olinear}
Assume $\chk\neq2$. Let $(M,q)$ be a quadratic space of dimension $2n$ or $2n+1$ over $k$. Let $(\cdot,\cdot):M\times M\to k$ be the associated symmetric bilinear form $(x,y)=q(x+y)-q(x)-q(y)$. The regular elliptic numbers of $m$ of the root systems of type $B_{n}, D_{n}$ and  $\leftexp{2}{D}_{n}$ are in bijection with
\begin{itemize}
\item Type $B_{n}$ ($\dim M=2n+1$): divisors $d|n$ (corresponding $m=2n/d$);
\item Type $D_{n}$ ($\dim M=2n$): even divisors $d|n$ (corresponding $m=2n/d$) or odd divisors $d|n-1$ (corresponding $m=2(n-1)/d$);
\item Type $\leftexp{2}{D}_{n}$ ($\dim M=2n$): odd divisors $d|n$ (corresponding $m=2n/d$) or even divisors $d|n-1$ (corresponding $m=2(n-1)/d$). 
\end{itemize}
We write $m=2\ell$. Fix a decomposition
\begin{equation}\label{sod}
M= M_{0}\oplus M_{1}\oplus\cdots M_{\ell-1}\oplus M_{\ell}\oplus M_{\ell+1}\oplus\cdots\oplus M_{m-1}
\end{equation}
where $\dim M_{i}=d$ for $i=1,\cdots, \ell-1, \ell+1,\cdots, m-1$, $\dim M_{0}$ and $\dim M_{\ell}$ are either $d$ or $d+1$, and we make sure that when $\dim M=2n+1$, $\dim M_{0}$ is even. We see that
\begin{itemize}
\item Type $B_{n}$: $\dim M_{0}$ is even and $\dim M_{\ell}$ is odd; 
\item Type $D_{n}$: $\dim M_{0}=\dim M_{\ell}$ is even;
\item Type $\leftexp{2}{D}_{n}$: $\dim M_{0}=\dim M_{\ell}$ is odd.
\end{itemize}
The decomposition \eqref{sod} should satisfy $(M_{i}, M_{j})=0$ unless $i+j\equiv 0\mod m$. The restriction of $q$ to $M_{0}$ and $M_{\ell}$ are denoted by $q_{0}$ and $q_{\ell}$. The pairing $(\cdot,\cdot)$ induce an isomorphism $M_{i}^{*}\cong M_{m-i}$.

Let $M_{+}=\oplus_{i>0}M_{i}$, then $M=M_{0}\oplus M_{+}$ and correspondingly $q=q_{0}\oplus q_{+}$. We use $q_{0,K}$ and $q_{+,K}$ to denote their $K$-linear extension to $M_{0}\otimes K$ and $M_{+}\otimes K$ respectively.  Define a new quadratic form $q'_{K}$ on $M\otimes K$ by
\begin{equation}\label{qK}
q'_{K}=\varpi q_{0,K}\oplus q_{+,K}.
\end{equation}
Let $G=\SO(M\otimes K, q'_{K})$. We have
\begin{itemize}
\item When $\dim M_{0}$ is even, choosing a maximal isotropic decomposition $M_{0}=M_{0}^{+}\oplus M_{0}^{-}$ gives a self-dual lattice $M_{0}^{+}\otimes\calO_{K}\oplus M_{0}^{-}\otimes\varpi^{-1}\calO_{K}\oplus M_{+}\otimes\calO_{K}\subset M\otimes K$. Therefore $G$ is a split orthogonal group.
\item When $\dim M_{0}$ is odd, $G$ is a quasi-split orthogonal group of type $\leftexp{2}{D}_{n}$.
\end{itemize}

Let $\wt\bP_{m}\subset G(K)$ be the stabilizer of the lattice chain $\L_{m-1}\supset \L_{m-1}\supset\cdots\supset\L_{0}$ under $G(K)$, where
\begin{equation*}
\L_{i}=\sum_{0\leq j\leq i}M_{j}\otimes\calO_{K}+\sum_{i<j\leq m-1}M_{j}\otimes\varpi\calO_{K}.
\end{equation*}
Its reductive quotient $\tL_{m}$ is the subgroup of $\Og(M_{0}, q_{0})\times \prod_{i=1}^{\ell-1}\GL(M_{i})\times\Og(M_{\ell}, q_{\ell})$ of index two, where the factor $\GL(M_{i})$ acts on $M_{i}$ by the standard representation and on $M_{m-i}=M_{i}^{*}$ by the dual of the standard representation.  We have 
\begin{equation}\label{soL}
\tLab_{m}\cong \{\pm1\}\times\prod_{i=1}^{\ell-1}\Gm
\end{equation}
by taking determinants (for elements in $\tL_{m}$, the determinant of the last factor $\Og(M_{\ell}, q_{\ell})$ is the same as the first one $\Og(M_{0},q_{0})$). The subgroup $\bP_{m}\subset\wt\bP_{m}$, defined as the kernel of $\wt\bP_{m}\to \tL_{m}\to\tLab_{m}\surj\{\pm1\}$, is an admissible parahoric subgroup of $G(K)$ with $m(\bP_{m})=m$.  The vector space $V_{m}:=V_{\bP_{m}}$ is
\begin{equation}\label{soV}
V_{m}=\Hom(M_{1},M_{0})\oplus\Hom(M_{2},M_{1})\oplus\cdots\oplus\Hom(M_{\ell}, M_{\ell-1}).
\end{equation}
Again, it is more natural to view $V_{m}$ as $\tau$-invariant cyclic quivers
\begin{equation}\label{oquiver}
\xymatrix{ & M_{1}\ar[dl]^{\psi_{0}} & M_{2}\ar[l]^{\psi_{1}} & \cdots\ar[l] & M_{\ell-1}\ar[l]^{\psi_{\ell-2}} \\
M_{0}\ar[dr]^{\psi_{m-1}} & & & & & M_{\ell}\ar[ul]^{\psi_{\ell-1}}\\
& M_{m-1}\ar[r]^{\psi_{m-2}} & M_{m-2}\ar[r] & \cdots\ar[r]^{\psi_{\ell+1}} & M_{\ell+1}\ar[ur]^{\psi_{\ell}} }
\end{equation}
where $\tau$ sends $\{\psi_{i}:M_{i+1}\to M_{i}\}$ to $\{-\psi_{m-1-i}^{*}: M_{i+1}=M_{m-1-i}^{*}\to M_{m-i}^{*}=M_{i}\}$ (indices are understood modulo $m$). 

When $m=2$, the quiver degenerates to a pair of maps
\begin{equation*}
\xymatrix{M_{0}\ar@<1ex>[r]^{\psi_{1}} & M_{1}\ar@<1ex>[l]^{\psi_{0}}}
\end{equation*}
such that $\psi_{1}=-\psi_{0}^{*}$.

\subsection{Stable locus} The dual space $V^{*}_{m}$ is the space of $\tau$-invariant cyclic quivers similar to \eqref{oquiver}, except that all the arrows are reversed.  Let $\phi_{i}: M_{i}\to M_{i+1}$ be the arrows. Then $\phi=(\phi_{0}, \cdots, \phi_{m-1})\in V^{*}_{m}$ is stable if and only if
\begin{itemize}
\item All the maps $\phi_{i}$ have the maximal possible rank;
\item We have two quadratic forms on $M_{0}$: $q_{0}$ and the pullback of $q_{\ell}$ to $M_{0}$ via the map $\phi_{\ell-1}\cdots\phi_{0}:M_{0}\to M_{\ell}$. They are in general position in the sense explained in \S\ref{ss:ustable}.
\end{itemize}

\subsection{The moduli stack}\label{omoduli}
We make a remark about the moduli stack $\Bun_{G}(\wt\bP_{0}, \bP^{+}_{\infty})$. It classifies the following data
\begin{enumerate}
\item A vector bundle $\calE$ over $X=\PP^{1}$ of rank equal to $\dim M$ with a perfect symmetric bilinear pairing on $\calE|_{\pline}$ (into $\calO_{\pline}$), such that the corresponding rational map $\iota: \calE\dashrightarrow\calE^{\vee}$ has simple pole at $0$ whose residue has rank $\dim M_{0}$ and simple zero at $\infty$. 
\item A filtration $\calE(-\{\infty\})=F_{-1}\calE\subset F_{0}\calE\subset F_{1}\calE\subset \cdots F_{m-1}\calE=\calE$ together with isomorphisms $\Gr^{F}_{i}\calE\cong M_{i}$. Extend this filtration by letting $F_{i+m}\calE:=F_{i}\calE(\{\infty\})$. Then we require that $F_{i}\calE$ and $F_{2m-1-i}\calE$ are in perfect pairing around $\infty$ under the quadratic form on $\calE$, such that the induced pairing between $M_{i}$ and $M_{m-i}$ is the same as the one fixed in \S\ref{olinear}.
\item A filtration $\calE(-\{0\})=F^{m}\calE\subset F^{m-1}\calE\subset \cdots\subset F^{1}\calE\subset F^{0}\calE=\calE$ such that $\Gr_{F}^{i}\calE$ has dimension $\dim M_{i}$. Extend this filtration by letting $F^{i+m}\calE:=F^{i}\calE(-\{0\})$. Then we require that $F^{i}\calE$ and $F^{1-i}\calE$ are in perfect pairing into around $0$ under the quadratic form on $\calE$. 
\item A trivialization of $\delta:\calO_{X}\cong\det\calE$ such that $\delta^{\vee}\circ\det(\iota)\circ\delta=t^{-\dim M_{0}}\in k[t,t^{-1}]=\Aut_{\pline}(\calO_{\pline})$ (which is unique up to a sign).
\end{enumerate}

\subsection{The scheme $\frG_{\l}$}\label{sofrG}
Let $\l\in\xcoch(T)$ be the dominant minuscule coweight such that $V_{\l}$ is the standard representation of the dual group $\dualG=\Sp_{2n}$ or $\SO_{2n}$. We shall describe the scheme $\frG_{\l}$ and the map $(f',f''):\frG_{\l}\to\tLab_{m}\times V_{m}$. 

Let $\calE_{+}=M_{+}\oplus\calO_{X}$ and $\calE_{0}=M_{0}\otimes\calO_{X}$. Let $q_{0,\calE}$ and $q_{+,\calE}$ denote the $\calO_{X}$-linear extension of $q_{0}$ and $q_{+}$ to $\calE_{0}$ and $\calE_{+}$. Define a new rational quadratic form on the trivial vector bundle $\calE=M\otimes\calO_{X}=\calE_{0}\oplus\calE_{+}$ by
\begin{equation}\label{tq}
q_{\calE}=t^{-1}q_{0,\calE}\oplus q_{+,\calE}.
\end{equation}
Then $(\calE,q_{\calE})$ satisfies the condition in \S\ref{omoduli}(1). The decomposition \eqref{sod} gives the filtrations required in \S\ref{omoduli}(2)(3). Fixing a trivialization of the line $\det M_{0}\otimes\det M_{\ell}$ (which is the same as $\det M$), then we get the data required in \S\ref{omoduli}(4). We thus get a point $(\calE, q_{\calE}, F_{*}, F^{*}, \cdots)\in\Bun_{G}(\wt\bP_{0}, \bP_{\infty}^{+})(k)$. This is the unique point of this moduli stack with trivial automorphism group.

The group ind-scheme $\frG$ is the group of orthogonal automorphisms of $\calE|_{X-\{1\}}$ preserving all the auxiliary data specified in \S\ref{omoduli}. The subscheme $\frG_{\l}$ consists of $g\in \frG\subset G(F)$ whose entries have at most simple poles at $t=1$, and  $\Res_{t=1}g$ has rank one. Therefore any element in $\frG_{\l}$ can be uniquely written as
\begin{equation*}
g=\frac{t}{t-1}A-\frac{1}{t-1}B
\end{equation*}
for $A,B\in\GL(M)$.

\begin{lemma}\label{l:sopre}
\begin{enumerate}
\item The scheme $\frG_{\l}$ classifies pairs $(A,B)\in\GL(M)\times\GL(M)$ of block form
\begin{equation}\label{AB0+}
A=\mat{I}{A_{0+}}{0}{A_{++}}; B=\mat{B_{00}}{0}{B_{+0}}{B_{++}}
\end{equation}
under the decomposition $M=M_{0}\oplus M_{+}$ and satisfying
\begin{itemize}
\item $A_{++}, B_{++}\in \Og(M_{+},q_{+}), B_{00}\in\Og(M_{0},q_{0})$ and
\begin{eqnarray}
\label{ooo1} (x,B_{00}x)=2q_{0}(x)+q_{+}(B_{+0}x);\\
\label{ooo2} (A_{++}y, B_{++}y)=q_{0}(A_{0+}y)+2q_{+}(y);\\
\label{ooo3} (x,A_{0+}y)=(B_{+0}x,A_{++}y);\\
\label{ooo4} (B_{00}x, A_{0+}y)=(B_{+0}x, B_{++}y)
\end{eqnarray}
for all $x\in M_{0}$ and $y\in M_{+}$.
\item $A_{++}$ lies in the unipotent radical $U(F_{*})$ of the parabolic $P(F_{*})\subset\SO(M_{+},q_{+})$ preserving the filtration $F_{*}$ of $M_{+}$; $B_{++}$ lies in the opposite parabolic $P(F^{*})\subset\Og(M_{+},q_{+})$ preserving the filtration $F^{*}$ of $M_{+}$;
\item $A-B$ has rank one.
\end{itemize}
\item The morphism $(f',f''): \frG_{\l}\to\tLab_{m}\times V_{m}$ is given by
\begin{eqnarray*}
f'(A,B)&=&(\det(B_{0,0}), \cdots, \det(B_{\ell,\ell}));\\
f''(A,B)&=&(A_{0,1}, \cdots, A_{\ell-1, \ell}) 
\end{eqnarray*}
in the block presentation of $A,B$ under the decomposition \eqref{sod}.
\item Consider the morphism $j:\frG_{\l}\to\PP(M)$ sending $(A,B)\in\frG_{\l}$ to the line in $M$ that is the image of the rank one endomorphism $C=I-A^{-1}B$. Then the image of $j$ is contained in the quadric $Q(q)$ defined by $q=0$, and the resulting map $j:\frG_{\l}\to Q(q)$ is an open embedding. 
\end{enumerate}
\end{lemma}
\begin{proof}
(1) The fact that $g$ preserves the quadratic form $q_{\calE}$ in \eqref{tq} is equivalent to the conditions in the first bulleted point. The rest of (1) and (2) are similar to their unitary or symplectic counterparts. 

(3) Write $C=I-A^{-1}B$ in the block form
\begin{equation*}
C=\mat{I-B_{00}+A_{0+}A^{-1}_{++}B_{+0}}{A_{0+}A_{++}^{-1}B_{++}}{-A_{++}^{-1}B_{+0}}{I-A^{-1}_{++}B_{++}}.
\end{equation*}
Since $I-A^{-1}_{++}B_{++}$ has rank one (it is easy to show that it is nonzero) and $A^{-1}_{++}B_{++}\in\Og(M_{+},q_{+})$, we must have $A^{-1}_{++}B_{++}=R_{[v_{+}]}$ (the orthogonal reflection in the direction of $v_{+}$) for some nonzero vector $v_{+}\in M_{+}$. In other words, $I-A^{-1}_{++}B_{++}=\frac{(-,v_{+})}{q_{+}(v_{+})}v_{+}$. Since $C$ has rank one, we must have $A_{0+}A_{++}^{-1}B_{++}=\frac{(-,v_{+})}{q_{+}(v_{+})}v_{0}$ for some vector $v_{0}\in M_{0}$. Hence $A_{0+}(x)=-\frac{(x,v_{+})}{q_{+}(v_{+})}v_{0}$. The relation \eqref{ooo2} implies $q_{0}(v_{0})+q_{+}(v_{+})=0$, i.e. $v=(v_{0},v_{+})$ satisfies $q(v)=0$, or $[v]\in Q(q)$. Note that $[v]$ is exactly the image of $C$, i.e., $j(A,B)=[v]$.

Using \eqref{ooo3} and \eqref{ooo4} we get $B_{00}=R_{[v_{0}]}$ and $A^{-1}_{00}B_{+0}(x)=\frac{(x,v_{0})}{q_{0}(v_{0})}v_{+}$. Therefore $[v]$ determines $B_{00}, B_{+0}, A_{0+}$ and $A^{-1}_{++}B_{++}$. Since $U(F_{*})\times P(F^{*})\incl \Og(M_{+},q_{+})$ is an open embedding, $[v]$ also uniquely determines $A_{++}\in U(F^{*})$ and $B_{++}\in P(F^{*})$, and $j$ is also an open embedding.
\end{proof}

As in Proposition \ref{p:ucal}, for $1\leq i\leq \ell$, we define  $q_{[i,m-i]}$ to be the restriction of $q$ to $M_{i}\oplus\cdots\oplus M_{m-i}$ and extended by zero to $M$ (so that $q_{+}=q_{[1,m-1]}$).

\begin{prop}\label{p:socal} Recall that the group $G=\SO(M\otimes K, q'_{K})$ is of type $B_{n}, D_{n}$ or $\leftexp{2}{D}_{n}$.
\begin{enumerate}
\item Under the open embedding $j:\frG_{\l}\incl Q(q)$, $\frG_{\l}$ is the complement of the union of the divisors defined by $q_{\ell}=0, q_{[\ell-1,\ell+1]}=0,\cdots, q_{[1,m-1]}=0$. The morphism $(f',f''): Q(q)\supset \frG_{\l}\to \tLab_{m}\times V_{m}$ is given by
\begin{eqnarray}
\label{sof'} f'([v])&=&(-1, \frac{q_{[1,m-1]}(v)}{q_{[2,m-2]}(v)}, \frac{q_{[2,m-2]}(v)}{q_{[3,m-3]}(v)}, \cdots, \frac{q_{[\ell-1,\ell+1]}(v)}{q_{\ell}(v)})\\
\label{sof''} f''([v])&=&(-\frac{(-,v_{m-1})}{q_{[1,m-1]}(v)}v_{0},-\frac{(-,v_{m-2})}{q_{[2,m-2]}(v)}v_{1}, \cdots, -\frac{(-,v_{\ell})}{q_{\ell}(v)}v_{\ell-1}).
\end{eqnarray}
Here we  write $v=(v_{0},\cdots, v_{m-1})$ under the decomposition \eqref{sod}.
\item The local system $\Kl^{\St}_{\dualG, \bP_{m}}(\chi)$ attached to the standard representation of $\dualG=\Sp_{2n}$ or $\SO_{2n}$, the admissible parahoric subgroup $\bP_{m}$ and the multiplicative character $\chi=(\chi_{0},\cdots,\chi_{\ell})$ (where $\chi_{0}$ has order two) is given by (the restriction to $V^{*,\st}_{m}$ of) the Fourier transform of the complex $f''_{!}f'^{*}\calL_{\chi}[\dim M-2](\frac{\dim M-2}{2})$.
\end{enumerate}
\end{prop}
\begin{proof} 
(1) In the proof of Lemma \ref{l:sopre} we have shown that if $j(A,B)=[v]$ and $v=(v_{0}, v_{+})$, then $A^{-1}_{++}B_{++}=R_{[v_{+}]}, B_{00}=R_{[v_{0}]}$ and $A_{0+}(x)=-\frac{(x,v_{+})}{q_{+}(v_{+})}v_{0}$.  In particular, $\det(B_{00})=-1$, which gives the first entry of $f'(A,B)$ given in \eqref{sof'}.

It remains to express $R_{[v_{+}]}$ as $A^{-1}_{++}B_{++}$ for $A\in U(F_{*})$ and $B\in P(F^{*})$. The procedure is the same as in the proof of Proposition \ref{p:ucal}. This gives all of \eqref{sof'} and \eqref{sof''} but the first entry of \eqref{sof''}. The first entry of $f''(A,B)$ is $A_{0,1}$, which is the first block of $A_{0+}=-\frac{(-,v_{+})}{q_{+}(v_{+})}v_{0}$. 

(2) We only need to apply Proposition \ref{p:Four} to the standard representation $V_{\l}$ of $\dualG$; note that $\IC_{\l}=\Ql[\dim M-2](\frac{\dim M-2}{2})$ in this case. 
\end{proof}

Similar to Corollary \ref{c:uKl}, we have
\begin{cor}\label{c:oKl}
Let $\phi=(\phi_{0}, \phi_{1}, \cdots, \phi_{m-1})\in V^{*,\st}_{m}(k)$ be a stable functional. Recall that $\frG_{\l}$ in this case is $Q(q)-\cup_{i=1}^{\ell}Q(q_{[i,m-i]})$. Let $f_{\phi}: \tpline\times \frG_{\l}\to \AA^{1}$ be given by
\begin{equation*}
f_{\phi}(x,[v])=-\frac{(\phi_{0}v_{0}, v_{m-1})}{q_{[1,m-1]}(v)}x-\sum_{i=1}^{\ell-1}\frac{(\phi_{i}v_{i}, v_{m-i-1})}{q_{[i+1, m-i-1]}(v)}.
\end{equation*}
Let $\pi:\tpline\times\frG_{\l}\to\tpline$ be the projection.
Then we have an isomorphism over $\tpline$
\begin{equation*}
\Kl^{\St}_{\dualG,m}(1,\phi)\cong\pi_{!}f_{\phi}^{*}\AS_{\psi}[\dim M-2]\left(\frac{\dim M-2}{2}\right).
\end{equation*}
\end{cor}

\section{Addendum to \cite{Ymotive}}

\begin{theorem} The Conjectures 5.10, 5.11 and 5.13 of \cite{Ymotive} are true. In particular, the local systems constructed in the main theorem of \cite{Ymotive} are cohomologically rigid.
\end{theorem}
In fact,  \cite[Conjecture 5.10]{Ymotive} can be proved in the same way as Theorem \ref{th:unip}. A proof of \cite[Conjecture 5.11]{Ymotive} will follow from results in the forthcoming work \cite{YZ}. Finally, by \cite[Lemma 5.14]{Ymotive}, \cite[Conjecture 5.13]{Ymotive} follows from the other two conjectures mentioned above.

\noindent{\bf Acknowledgements} The author thanks  B. Gross and J-K. Yu for inspiring conversations.

\end{document}